\documentclass[11pt]{article}

\usepackage[T1]{fontenc}

\usepackage{geometry}
\usepackage{setspace}
\usepackage{times}
\usepackage{graphicx}
\graphicspath{{Figures/}}
\usepackage{bm}
\usepackage{amssymb,amsmath,amsthm}
\usepackage{booktabs}
\usepackage{todonotes}
\usepackage{bbm}
\usepackage{caption}
\usepackage{multirow}
\usepackage{natbib}
\usepackage{authblk}
\usepackage{hyperref}

\renewcommand{\P}{\mathbb{P}}

\newcommand{\R}{\mathbb{R}}

\newcommand{\eqdis}{\stackrel{\text{\tiny d}}{=}}

\newcommand{\Ind}[1]{I_{\{#1\}}}

 \newcommand{\rd}{\mathrm{d}}
 
 \newcommand{\indep}{\perp \!\!\! \perp}
\renewcommand{\leq}{\leqslant}

\renewcommand{\le}{\leqslant}

\newcommand{\udp}{T}

\newcommand{\udponAl}{\udp\vert_{A_\ell}}
\newcommand{\wtfunc}{\omega}
\newcommand{\unionA}{A^{\scriptscriptstyle \cup}}
\newcommand{\party}[1]{b_{#1}}
\newcommand{\partyU}{B}

\newcommand{\newtext}[1]{\textcolor{black}{#1}}

\newtheorem{proposition}{Proposition}
\newtheorem{theorem}{Theorem}

\newtheorem{corollary}{Corollary}
\newtheorem{definition}{Definition}

\begin{document}

\newcommand{\myThanks}{\thanks{Address correspondence to \texttt{alexander.mcneil@york.ac.uk}.}}

\title{Stochastic Inversion of Multivariate Uniform-Distribution-Preserving Transformations}

\author{ALEXANDER J.\ MCNEIL\myThanks}
\affil{The School for Business and Society, University of York,
  UK}

\author{JOHANNA G.\ NE\v{S}LEHOV\'A}
\affil{Department of Mathematics and Statistics, McGill University,
  Montr\'eal, Canada}

\maketitle

\begin{abstract}
  A multivariate transformation of the unit cube with component
  transformations that are piecewise continuously differentiable and
  uniform distribution preserving (udp) is considered. A stochastic inverse transformation
is defined using randomization to overcome the non-injective nature
of the udp
transformations. The inverse transformation
preserves the uniform
margins of a random vector distributed according to a copula and
yields different copulas for different randomizations.
A copula density transformation result for the multivariate stochastic
inverse is proved and illustrated in the bivariate case.
\end{abstract}


\section{Introduction}\label{sec:introduction}

Let
$T_1,\ldots,T_d$ be uniform distribution preserving (or
Lebesgue measure preserving)
transformations on the unit interval $[0,1]$, meaning that $T_j(U)\sim \mathcal{U}(0,1)$ whenever $U\sim\mathcal{U}(0,1)$, and consider the
mapping of the $d$-dimensional unit cube given by
$$
\mathcal{T}:[0,1]^d \to [0,1]^d, \quad
\bm{u}=(u_1,\ldots,u_d)^\top \mapsto
\mathcal{T}(\bm{u})=(T_1(u_1),\ldots,T_d(u_d))^\top.
$$
When the mapping $\mathcal{T}$ is applied to a random vector
$\bm{U}=(U_1,\ldots,U_d)^\top$ with uniform margins, we obtain a
random vector $\mathcal{T}(\bm{U})$ with uniform
margins, so that $\mathcal{T}$ induces a mapping between sets
of copulas, i.e.~multivariate distribution functions with uniform
margins. \newtext{A number of authors consider special cases of $\mathcal{T}$
and investigate the  copula families obtained from
$\mathcal{T}(\bm{U})$ when $\bm{U}$ has a given copula $C_{\bm{U}}$~\citep{nasri:2020,quessy:2024,hofert/pang:2025}.} In this article we consider the inverse
construction, where we infer models for $\bm{U}$ from models for
$\bm{V} = \mathcal{T}(\bm{U})$. Since the
component transformations $T_j$ may be non-injective, the
inverse construction involves a randomization and is referred to as multivariate
stochastic inversion. Our main result is to derive the possible joint density
functions for $\bm{U}$ when the joint density $c_{\bm{V}}$ of $\bm{V}$ is
given and the component
transformations of $\mathcal{T}$ are piecewise
continuously differentiable on finite partitions
of $[0,1]$.

Our interest in multivariate stochastic inversion is
motivated by the problem of constructing
parametric copula models for non-monotonic dependencies.
Consider two
continuous variables $X$ and $Y$ with distribution functions (dfs)
  $F_X$ and $F_Y$. According to Sklar's representation theorem (see, for
  example,~\cite{bib:joe-15}), the dependence between $X$ and $Y$
  can be described by their unique copula, the df $C_{\bm{U}}$ of $(U_1,U_2) = (F_X(X), F_Y(Y))$.
Suppose that the dependence between $X$ and $Y$ is characterized by strong positive dependence between $g(X)$ and $h(Y)$ for two piecewise continuous and strictly monotonic
functions $g$ and $h$ with \newtext{finitely many pieces. The
first three panels in the top row of Figure~\ref{fig:motivation} show
simulated datasets illustrating this situation:
$X$ and $Y$ have standard normal marginal distributions and
unusual dependence structures that are
progressively more difficult to model with existing
copulas as we go from left to right: the corresponding copula data $(U_1,U_2)$
are shown in the row below.
In each case we can find functions $g$ and $h$ such that
$(g(X),h(Y))$ are as shown in the final panel of the top row where they exhibit a positive
dependence relationship that is much easier to model with a copula.}

\begin{figure}[t]
   \centering
   \includegraphics[width=16.5cm,height=4cm]{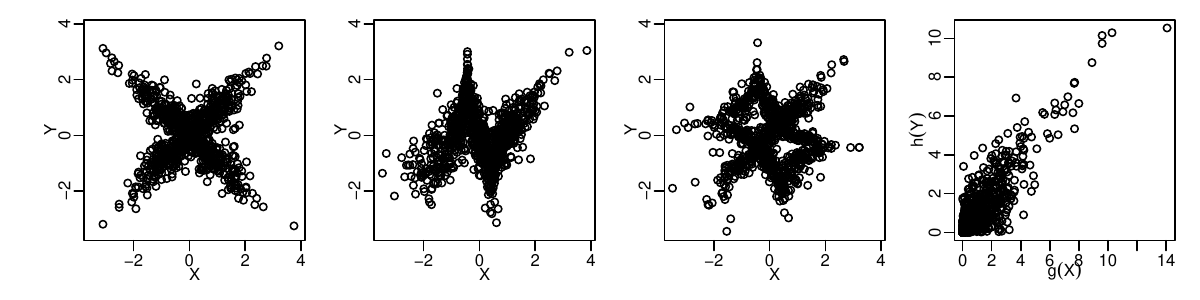}\\
     \includegraphics[width=16.5cm,height=4cm]{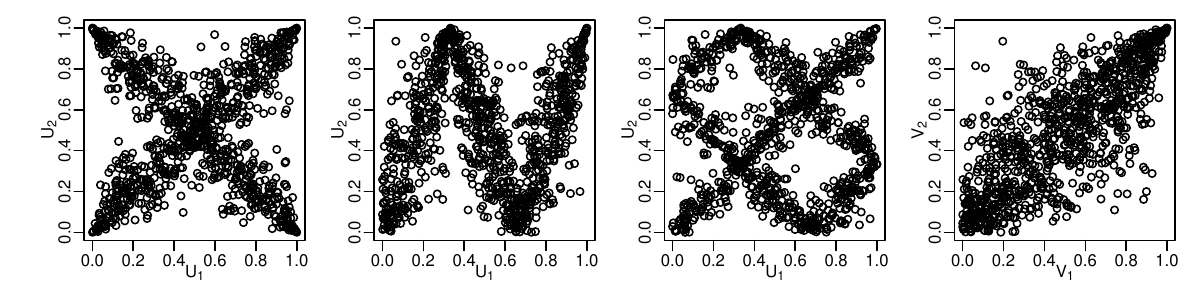}
  \vspace{-0.5cm}
   \caption{\label{fig:motivation} Top row, first three panels: 1000 realizations of
     standard normal variables $X$ and $Y$ in models with normal
     margins and non-standard copulas. Top row, final panel:
     corresponding realizations of $g(X)$ and $h(Y)$; these have chi-squared margins and
     a Gumbel copula. Bottom row: copula data corresponding to each
     panel of the top row. Details of $g$ and $h$ are found
     in Appendix~\ref{appendix:A}.}
 \end{figure}

 \newtext{
The copula $C_{\bm{V}}$ of $(g(X),h(Y))$ is the joint distribution
 of $(V_1,V_2) = (T_1(U_1),T_2(U_2))$
for two uniform-distribution-preserving
(udp) functions $T_1$ (determined by $g$ and $F_X$) and $T_2$
(determined by $h$ and $F_Y$). The copula data $(V_1,V_2)$ in the bottom
right panel of Figure~\ref{fig:motivation} are identical for all three
datasets.
If we can find an appropriate copula $C_{\bm{V}}$ for these data, the multivariate stochastic
inversion result in this article can be used to construct 
a possible model $C_{\bm{U}}$ for $(U_1,U_2)$
from $C_{\bm{V}}$. When the
functions $g$ and $h$ inducing strong monotonic dependence
between $X$ and $Y$ are not known a priori, the ACE
algorithm~\citep{bib:breiman-friedman-85} or a procedure
proposed in~\citet{bib:mcneil-neslehova-smith-25} may be used to
estimate them from data.}

\newtext{While the examples of Figure~\ref{fig:motivation} are
  simulated, empirical datasets showing v-shaped, x-shaped (as in
  first column), N-shaped (as in second column)
  and w-shaped dependence patterns have also been documented in
  the applied statistical and economic literature. Some examples are
  given in~\cite{bib:mcneil-neslehova-smith-25}.}


\newtext{Many of the basic results for udp functions can be found
in~\cite{bib:porubsky-salat-strauch-88}.
The concept of the univariate stochastic inverse is developed in~\cite{bib:mcneil-20}
and~\cite{bib:bladt-mcneil-21} for a class of v-shaped udp
functions known as v-transforms. 
\cite{bib:mcneil-neslehova-smith-25} define a stochastic
inverse for all sufficiently regular udp transformations 
and explain
how the multivariate generalization may be used to construct copula
models that attain the feasible values for certain measures of
non-monotonic dependence.}

\newtext{The aim of this article is to present
a general density transformation result for the multivariate
case. This result permits parametric models to be fitted to empirical
data showing non-monotonic dependencies using likelihood-based
methods. The potential of the method is illustrated in
Section~\ref{sec:applications} with an application to financial time series data
  showing similar x-shaped behaviour to the first column of Figure~\ref{fig:motivation}.}

\vspace{-0.5cm}

\section{Stochastic Inversion}\label{sec:stochastic-inversion}


\subsection{Regular udp functions and their stochastic inverses}\label{sec:regul-udp-funct}
\begin{definition}\label{def:udp-regular}
A measurable map $\udp : [0,1] \to [0,1]$ is called {\em uniform distribution preserving (udp)} if $T(U) \sim \mathcal{U}(0,1)$ for $U \sim \mathcal{U}(0,1)$.
A udp transformation is called {\em regular} if there exists a finite partition $a_0 = 0 < a_1 < \dots < a_{L-1} < a_L = 1$ for $L \in \mathbb{N}_1=\{1, 2, \dots\}$ such that $\udp$ is continuously differentiable on $A_\ell =(a_{\ell-1}, a_{\ell})$ for all $\ell \in \{1,\dots, L\}$. 
\end{definition}

\noindent Let $\udponAl$ denote the restriction of a regular udp $T$ to
$A_\ell$, let
$R_\ell = T(A_\ell) = \udponAl(A_\ell)$ denote the image of $A_\ell$
under $\udp$, and let $\unionA = \cup_{\ell=1}^L A_\ell$ denote the
union of open intervals on which $\udp$ is continuously
differentiable. \newtext{Let $\party{0} < \dots  < \party{K}$ be the  ordered distinct elements of the set
  $\{T_{+}(a_{\ell-1}), T_{-}(a_\ell), \ell = 1,\dots, L\}$ where
  $T_{+}$ and $T_{-}$ denote the one-sided limits of $T$ and $K \in
  \mathbb{N}_1$. Following \cite{bib:mcneil-neslehova-smith-25},
  $\party{0}=0$, $\party{K}=1$ and $\party{0} < \cdots <\party{K}$ is the so-called $T$-partition of $[0,1]$ such that, for all $k$, $(\party{k-1},\party{k})  \subseteq R_\ell$ for at least one $\ell$. Note that $\unionA$ and $\partyU = \cup_{k=1}^K(\party{k-1},\party{k})$ have Lebesgue measure one and hence also $T(\unionA) \supseteq B$.}


\newtext{For example, Figure~\ref{fig:Tplots} shows regular udp functions
  based on shifted Legendre polynomials following a construction
in~\citet{bib:mcneil-neslehova-smith-25}. The partitions on which the functions are
continuously differentiable have 2, 5, 6, 13 and 12 pieces,
respectively; the $T$-partitions are shown by dashed horizontal lines.}

 \begin{figure}[htb]
    \centering
     \includegraphics[width=16cm,height=4cm]{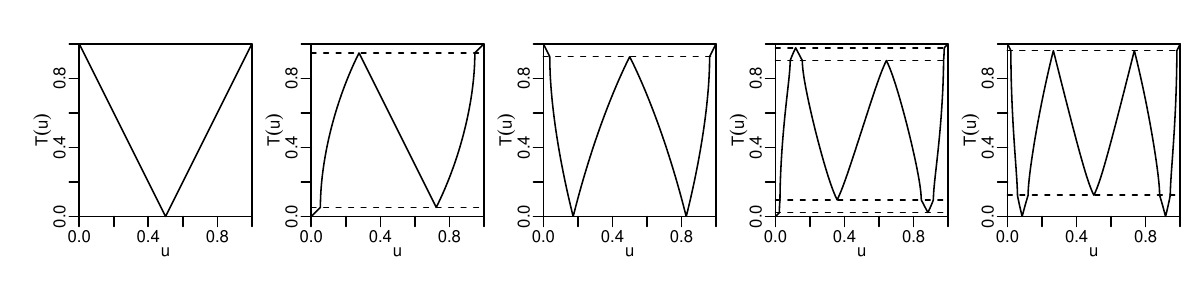}\vspace{-0.5cm}
   \caption{\label{fig:Tplots}  Plots of regular udp functions
 based on shifted
Legendre polynomials with degree $2,
\ldots,6$. The function $T(u) = |2u-1|$  in the first panel is
a v-transform.}
 \end{figure}

\newtext{The v-transform $T(u) = |2u-1|$ in the first panel of
Figure~\ref{fig:Tplots} is a very simple example of a udp
transformation and the one that is applied to the copula data $(U_1,U_2)$ in the first
column of Figure~\ref{fig:motivation} to transform
them into the copula data $(V_1,V_2)$ in the final column. To pass backwards between
these copulas we need the stochastic inverse of $T$. At a point $v\in (0,1)$ this means randomly selecting one of
$(1-v)/2$ and $(1+v)/2$ with probability $1/2$.}
\newtext{In general, stochastic
inversion of a udp transformation is more complicated and requires the
following definition from~\citet{bib:mcneil-neslehova-smith-25}. }


 \begin{definition}[Stochastic inverse of a regular udp transformation $\udp$]\label{def:si}
Let $\udp$ be a regular udp transformation and $Z \sim \mathcal{U}(0,1)$.
\newtext{For $v \in \partyU$, assign the inverse value 
$
\udp^\leftarrow(v,Z) = G^{-1}(Z,v)
$
where $G^{-1}(\cdot,v)$ is the generalized inverse of the distribution function $G(\cdot,v)$ of the discrete random variable with support $\udp^{-1}(\{v\})= \{u : \udp(u) = v, u \in \unionA\} = \{r_1(v),
\ldots, r_{n(v)}(v)\}$ and probabilities $1/|\udp^\prime(r_i(v))|$, $i
\in \{1,\dots, n(v)\}$.  Otherwise, for $v \not \in\partyU$, set
$\udp^\leftarrow(v,Z)$ to be $0$.}
\end{definition}


\newtext{The essential idea is that, on
all but a null set of points, we construct the inverse by finding the
finite set of
roots $\{r_1(v),
\ldots, r_{n(v)}(v)\}$ of $\udp(u) = v$  in $ \unionA$  and selecting one
of them by multinomial sampling according to a set of probabilities
$\{|\udp^\prime(r_1(v))|^{-1}, \ldots,
|\udp^\prime(r_{n(v)}(v))|^{-1}\}$. \citet{bib:mcneil-neslehova-smith-25} show that $ \{r_1(v),
\ldots, r_{n(v)}(v)\}$ is indeed non-empty and $\{|\udp^\prime(r_1(v))|^{-1}, \ldots,
|\udp^\prime(r_{n(v)}(v))|^{-1}\}$ are well-defined and sum up to
one. They also prove that, for almost all $v$, $\udp(
\udp^\leftarrow(v,Z)) = v$ a.s.\ and they establish the key property 
of the stochastic inverse, namely that it preserves uniformity in the following sense.}

\begin{theorem}\label{theorem:stochinverse}
If $\udp$ is a regular udp and $V\sim \mathcal{U}(0,1)$, $Z \sim \mathcal{U}(0,1)$ are independent,
 $\udp^\leftarrow(V,Z) \sim \mathcal{U}(0,1)$.
\end{theorem}

\newtext{For a point $v \in \partyU$, multinomial sampling may be thought of
as random allocation of its inverse to one of the intervals
$A_\ell$,  $\ell \in \{1,\ldots,L\}$.
Writing $U =
\udp^\leftarrow(V,Z)$, the corresponding allocation probabilities are
\begin{equation}
  \label{eq:15}
  p_\ell(v) = \P(U \in A_\ell \mid V = v) = G(a_\ell,v) -
  G(a_{\ell-1},v) = \frac{\Ind{v \in R_\ell}}{
    |\udp^\prime(\udponAl^{-1}(v))|}.
\end{equation}
For $v \not \in \partyU$,  we set $p_\ell(v)=0$ for all $\ell \in \{1,\dots, L\}$; this is inconsequential because $[0,1]\setminus B$ is a null set.  Equation~\eqref{eq:15} implies
the identity $p_\ell(T(x)) = \left|T^\prime(x)\right|^{-1}$ for $x \in
A_\ell$ with $T(x) \in B$, used in the sequel.}


 
 

\subsection{Multivariate stochastic inversion}\label{sec:mult-stoch-invers}
\cite{bib:mcneil-neslehova-smith-25} formulate the following
Corollary of Theorem~\ref{cor:stochinversemult}.
\begin{corollary}\label{cor:stochinversemult}
Let $\bm{V} = (V_1,\dots, V_d)^\top$ and $\bm{Z} = ( Z_1,\dots,
Z_d)^\top$ be vectors of possibly dependent standard uniform random variables such
that $V_i$ is independent of $Z_i$ for $i=1,\ldots,d$. Let
$\udp_{1},\dots, \udp_{d}$ be regular udp transformations. Then
$\bm{U} = \left(\udp_{1}^\leftarrow(V_1,Z_1), \dots,
  \udp_{d}^\leftarrow(V_d,Z_d)\right)^\top$ is a vector of standard
uniform  variables and $\mathcal{T}(\bm{U}) = (\udp_{1}(U_1), \dots, \udp_{d}(U_d))^\top= \bm{V}$, almost surely.
\end{corollary}
\noindent
For $i=1,\ldots,d$, denote the partition on
which $T_i$ is regular by
$a_{i0} = 0 < a_{i1} < \dots < a_{i(L_i-1)} <
a_{i,L_i} = 1$ for $L_i \in \{1,2,\ldots\}$.
 As in
Definition~\ref{def:si}, write
$G_{i}(\cdot,v)$ for the multinomial distribution function associated
with the stochastic inverse of $T_i$ at $v$.
For $\ell \in
\{1,\ldots,L_i\}$, let $A_{i \ell} = (a_{i(\ell-1)},
a_{i\ell})$ and $R_{i \ell}
= T_i(A_{i \ell})$ and set $\unionA_i =
A_{i 1} \cup \cdots \cup A_{iL_i}$. \newtext{Let $0=\party{i0} < \dots < \party{iK_i}=1$ be the $T_i$-partition, $\partyU_i = \cup_{k=1}^{K_i}(\party{i (k-1)}, \party{ik})$,
and for $v \in B_i$ let $p_{i\ell}(v) = \P(U_i \in A_{i\ell} \mid V_i = v)$ 
be as in~\eqref{eq:15}. }
 
 \newtext{Analogous to~\eqref{eq:15}, let
$p_{\ell_1,\ldots,\ell_d}(\bm{v}) = \P(\bm{U} \in A_{1\ell_1} \times \cdots \times
  A_{d \ell_d} \mid \bm{V} = \bm{v})$
denote the conditional probability of allocation to the cell $A_{1\ell_1} \times \cdots \times
A_{d \ell_d}$ given $\bm{V}$, for $\ell_i \in
\{1,\ldots,L_i\}$ and $i\in\{1,\ldots,d\}$. When $v_i \not \in \partyU_i$ for at least one $i$, we set $p_{\ell_1,\ldots,\ell_d}(\bm{v})=0$, noting that such points are inconsequential since  $\partyU_1 \times \dots \times \partyU_d \subset (0,1)^d$ has Lebesgue measure one. In multivariate stochastic inversion the points
$\bm{v} \in \partyU_1 \times \cdots \times\partyU_d$ are randomly allocated to an inverse value
$\bm{u} = (T_1^\leftarrow(v_1,Z_1),\ldots,
T_d^\leftarrow(v_d,Z_d))^\top$ in a box $A_{1 \ell_1} \times \cdots \times
A_{d \ell_d} \in (0,1)^d$  by multinomial sampling.} The probabilities $p_{\ell_1,\ldots,\ell_d}(\bm{v})$ define an $L_1 \times \dots \times L_d$ contingency
table or $d$-variate multinomial distribution with marginal
multinomial distributions given by $\{p_{i\ell_i}(v_i), \ell_i =1,\ldots,L_i\}$
for $i\in\{1,\ldots,d\}$. They 
depend on the conditional distribution $F_{\bm{Z} \mid \bm{V}}$ as follows.
\begin{proposition}\label{prop:multinomialp}
\newtext{For $\bm{v} \in \partyU_1 \times \dots \times \partyU_d$,
  the allocation probabilities are given by 
$  p_{\ell_1,\ldots,\ell_d}(\bm{v}) =  \P(Z_i\in (G_{i}(a_{i(\ell_i-1)},v_i),
      G_i(a_{i\ell_i},v_i)],$ $i=1,\ldots,d \mid
      \bm{V} = \bm{v})$ if  $\bm{v} \in  R_{1\ell_1} \times \dots \times R_{d\ell_d}$ and by $0$ otherwise.}
\end{proposition}
\begin{proof}
\newtext{For a point $\bm{v} \in (B_1 \cap R_{1\ell_1}) \times \cdots \times (B_d \cap R_{d\ell_d})$ observe that $p_{\ell_1,\ldots,\ell_d}(\bm{v})$ equals}
  \begin{multline*}
\P\left( \udp_i^\leftarrow(V_i,Z_i) \in A_{i\ell_i}, i = 1,\ldots,d 
      \mid \bm{V} = \bm{v} \bigr) =  \P\bigl(G_{i}^{-1}(Z_i,v_i) \in A_{i\ell_i} ,i=1,\ldots,d\mid
     \bm{V} = \bm{v} \right)     \\
    = \P\left(Z_i\in \left(G_{i}(a_{i(\ell_i-1)},v_i),
      G_i(a_{i\ell_i},v_i)\right], i=1,\ldots,d \mid
      \bm{V} = \bm{v} \right),
  \end{multline*}
\newtext{while a point $\bm{v} \not\in (B_1 \cap R_{1\ell_1}) \times \cdots \times (B_d \cap R_{d\ell_d})$ cannot be randomly allocated to $A_{1\ell_1}\times \cdots \times
A_{d\ell_d}$.}
\end{proof}

We now describe the distribution of the
multivariate stochastic inverse when $\bm{V}$
has a joint density. Henceforth, we use expressions for probability densities like
$c_{\bm{V}}(v_1,\ldots,v_d)$ and $c_{\bm{V}}(\bm{v})$ interchangeably.

\begin{theorem}\label{theorem:mult-stoch-invers}
  Let $\bm{V} = (V_1,\dots, V_d)^\top$ and $\bm{Z} = ( Z_1,\dots,
Z_d)^\top$ be vectors of possibly dependent standard uniform random variables such
that $V_i$ is independent of $Z_i$ for $i\in\{1,\dots,d\}$. Let
$\udp_{1},\dots, \udp_{d}$ be regular udp transformations. If $\bm{V}$ is distributed
  according to a copula with density $c_{\bm{V}}$ then
the
  multivariate stochastic inverse
  $\bm{U} = (\udp_{1}^\leftarrow(V_1,Z_1), \dots,
  \udp_{d}^\leftarrow(V_d,Z_d))^\top$ is distributed according to a
  copula with density given, for almost all $\bm{u} \in (0,1)^d$, by $c_{\bm{U}}(u_1,\ldots,u_d) =
    c_{\bm{V}}(T_1(u_1),\ldots,T_d(u_d)) \wtfunc(u_1,\ldots,u_d)$ where
  \begin{equation}\label{eq:4}
    \wtfunc(u_1,\ldots,u_d) =
  \sum_{\ell_1 =1}^{L_1}\cdots  \sum_{\ell_d =1}^{L_d}\Ind{\bm{u} \in A_{1\ell_1}\times \cdots
  \times A_{d\ell_d}} p_{\ell_1,\ldots,\ell_d}(\mathcal{T}(\bm{u}))   \prod_{i=1}^d \left|
      T_i^\prime(u_i)\right|   .
\end{equation}
  \end{theorem}
  \begin{proof}
First, we show that for arbitrary indices $\ell_i \in \{1,\dots, L_i\}$, $i \in \{1,\dots, d\}$ and any $t_i \in (a_{i(\ell_i-1)},a_{i \ell_i}]$,
\begin{equation}\label{eq:lemma}
\P(U_i \in (a_{i (\ell_{i}-1)}, t_i), i \in \{1,\dots, d\}) = \int_{a_{1(\ell_1-1)}}^{t_1} \dots\int_{a_{d(\ell_d-1)}}^{t_d}c_{\bm{U}}(x_1,\dots, x_d) dx_1 \dots d x_d,
\end{equation}
where $c_{\bm{U}}$ has the form claimed in Theorem~\ref{theorem:mult-stoch-invers}. To this end, we first calculate that
  \begin{multline*}
    \P\left(U_i \in (a_{i(\ell_i-1)}, t_i), i=1,\ldots,d\right) 
    = \P\left(
  \udp_i^\leftarrow(V_i,Z_i) \in (a_{i(\ell_i-1)}, t_i),
                                                            i=1,\ldots,d
                                                            \right)
    \\
    = \int_{R_{1\ell_1}}\cdots\int_{R_{d\ell_d}}
      c_{\bm{V}}(\bm{v}) \P\left(G_{i}^{-1}(Z_i,v_i) \in
      (a_{i(\ell_i-1)}, t_i), i = 1,\ldots,d \mid
      \bm{V} = \bm{v} \right)        
      \rd \bm{v} \\
    = \int_{R_{1\ell_1}}\cdots\int_{R_{d\ell_d}}
      c_{\bm{V}}(\bm{v}) \P\left(Z_i\in \left(G_i(a_{i(\ell_i-1)},v_i),
      G_i(t_i, v_i)
      \right], i=1,\ldots,d\mid
      \bm{V} = \bm{v} \right)        
      \rd \bm{v}.
  \end{multline*}
  Because for each $i \in \{1,\dots, d\}$ and $\ell_i \in \{1,\dots,
 L_{i}\}$, the restriction $T_{i}\vert_{A_{i \ell_i}}$ is strictly monotonic and
 continuously differentiable on $A_{i \ell_i}$, we can make a change
 of variable $v_i \mapsto \udp_{i}(x_i)$ to obtain 
 \begin{multline*}
 \P\left(U_i \in (a_{i(\ell_i-1)}, t_i), i=1,\ldots,d\right) 
   =\\
     \int_{A_{1\ell_1}}\cdots\int_{A_{d\ell_d}}
      c_{\bm{V}}(\mathcal{T}(\bm{x})) \frac{\P\left(   Z_i\in
          \left(G_i(a_{i(\ell_i-1)}, T_i(x_i)),
      G_i(t_i, T_i(x_i))
      \right], i=1,\ldots,d    \mid
      \bm{V} = \mathcal{T}(\bm{x}) \right) }{   \prod_{i=1}^d \left| T_i^\prime(x_i)\right|^{-1}    }
      \rd \bm{x}.
 \end{multline*}
For points $\bm{x} = (x_1,\ldots,x_d)^\top$ and $\bm{t} =
(t_1,\ldots,t_d)^\top$ in $
A_{1\ell_1} \times \cdots \times A_{d \ell_d}$, it may be inferred
from~\eqref{eq:15} that 
$$
G_i(t_i, T_i(x_i)) = G_i(a_{i(\ell_i-1)},T_i(x_i)) + |T_i^\prime(x_i)|^{-1}\Ind{x_i
    \leq t_i}.
  $$
If $x_i > t_i$ for any $i \in
 \{1,\ldots,d\}$, then 
 $G_i(t_i, T_i(x_i)) = G_i(a_{i(\ell_i-1)}, T_i(x_i))$ and hence
 $$
 \P\left( Z_i\in \left(G_i(a_{i(\ell_i-1)}, T(x_i)),
      G_i(t_i, T_i(x_i))
      \right], i=1,\ldots,d  \mid
   \bm{V} = \mathcal{T}(\bm{x}) \right) =0.
 $$
Moreover $G_{i}(t_i,T_i(x_i))
    = G_i(a_{i \ell_i}, T_i(x_i)) $ for $x_i \leq t_i$.Using Proposition~\ref{prop:multinomialp} and the fact that all indicators in the summands on the right-hand side of \eqref{eq:4} vanish except one, 
 we obtain \eqref{eq:lemma} as claimed.
Now fix an arbitrary $\bm{u}=(u_1,\dots, u_d)^\top \in (0,1)^d$. Using the law of total probability, the fact that the margins $U_i$ are uniformly distributed and that $\unionA_i$ has Lebesgue measure $1$,
 \[
 \P(U_1\le u_1,\dots, U_d \le u_d) = \sum_{\ell_1 =1}^{L_1} \dots \sum_{\ell_d=1}^{L_d} \P(U_i \in (0, u_i) \cap A_{\ell_i}, i \in \{1,\dots, d\}).
 \]
 The summands vanish unless all intersections $(0, u_i) \cap A_{\ell_i}$ are non-empty, that is $u_i > a_{i(\ell_i-1)}$ for all $i \in \{1,\dots, d\}$. In the latter case,  $(0, u_i) \cap A_{\ell_i} = (a_{i(\ell_i-1)}, \min(u_i,a_{i \ell_i}))$. Applying identity \eqref{eq:lemma} gives
 \[ 
 \P(U_i \in (0, u_i) \cap A_{\ell_i}, i \in \{1,\dots, d\}) = \int_{a_{1(\ell_1-1)}}^{\min(u_1,a_{1 \ell_1})}\dots \int_{a_{d(\ell_d-1)}}^{\min(u_d,a_{d \ell_d})} c_{\bm{U}}(x_1,\dots, x_d) dx_1 \dots d x_d.
 \]
Gathering terms then shows that $c_{\bm{U}}$ is indeed the Lebesgue density of $\bm{U}$, as claimed.
\end{proof}

\newtext{The weighting function $\wtfunc$ in~\eqref{eq:4} is key to the general
result. It may be noted that, at a point $\bm{u} \in
A_{1\ell_1}\times \cdots \times A_{d\ell_d}$ such that $T_i(u_i) \in B_i$ for all $i \in \{1,\dots, d\}$,
$\wtfunc(u_1,\ldots,u_d)$ is equal to the
ratio of probabilities
$p_{\ell_1,\ldots,\ell_d}(\mathcal{T}(\bm{u}))/\big(\prod_{i=1}^d
p_{i\ell_i}(T_i(u_i))\big)$, where we have used the identity $\left|T_i^\prime(u_i)\right|^{-1} = p_{i\ell_i}(T_i(u_i))$. 
The numerator and denominator are both multinomial
probabilities of allocation to the cell 
$A_{1\ell_1}\times \cdots \times A_{d\ell_d}$, given $\bm{V} =
\mathcal{T}(\bm{u})$. The former depends on the conditional randomizer
distribution $F_{\bm{Z}\mid \bm{V}}$ and the latter corresponds to
the case where $Z_1,\ldots,Z_d$ are iid uniform
variables, independent of $\bm{V}$, so that
$$p_{\ell_1,\ldots,\ell_d}(\mathcal{T}(\bm{u})) =
  \prod_{i=1}^d \P\left(G_i(a_{i(\ell_i-1)}, T(u_i)) < Z_i 
\le G_i(a_{i \ell_i}, T_i(u_i)) \right)
 =\prod_{i=1}^dp_{i\ell_i}(T_i(u_i)).$$
In this case, which we label independent stochastic inversion, we
obtain the equations
$\wtfunc(u_1,\ldots,u_d)=1$ and  $c_{\bm{U}}(u_1,\ldots,u_d) =
    c_{\bm{V}}(T_1(u_1),\ldots,T_d(u_d))$.}
The following ancillary result, proved in Appendix~\ref{appendix:B}, shows that $\wtfunc$ is actually
  a probability density 
  and,
  under certain conditions, a copula density.
\begin{theorem}\label{prop:densityresult}
The function $\wtfunc$ given by~\eqref{eq:4} is a probability density
function. If $Z_i$ is independent of $\bm{V}$ for all $i\in\{1,\ldots,d\}$, it
is the density of a copula.
\end{theorem}

\newtext{The multinomial structure of $p_{\ell_1,\dots,\ell_d}(\bm{v})$ has implications
  for estimating densities of the form
  $c_{\bm{U}}(\bm{u}) =
    c_{\bm{V}}(T_1(u_1),\ldots,T_d(u_d)) \wtfunc(\bm{u})$ 
   from observed data
    $\{\bm{u}_1,\ldots,\bm{u}_n\}$. Suppose the udp
  functions $T_1,\ldots,T_d$ are known (or have been estimated).  We can construct data for $\bm{V}$ via the transformation
  $\bm{v}_j =\mathcal{T}(\bm{u}_j)$. For almost every
  value $\bm{v}_j$ we also obtain a 
  single multinomial observation of the cell $A_{1\ell_{1j}}\times \cdots
  \times A_{d\ell_{dj}}$ containing $\bm{u}_j$ so that determining $p_{\ell_1,\dots,\ell_d}(\bm{v})$
resembles a multinomial regression problem with a single
  outcome for each $\bm{v}_j$. However,
  $p_{\ell_1,\dots,\ell_d}(\bm{v})$ is difficult to parameterize
  directly because it is subject to complex constraints arising from
  the fact that the marginal multinomial probability functions
  $p_{1\ell_1}(v_1),\ldots,p_{d\ell_d}(v_d)$ are fixed by
  $T_1,\ldots,T_d$; an example that makes this clear is given in Appendix~\ref{appendix:C}. It is easier to parameterize models by specifying
  $F_{\bm{Z}\mid\bm{V}}$. However, to avoid identifiability problems, we
  have to be mindful that Proposition~\ref{prop:multinomialp}
  implies that different randomizer models $F_{\bm{Z}\mid \bm{V}}$ may
  lead to the same density $c_{\bm{U}}$ provided they coincide at a
  finite number of points.}
 


\section{The bivariate case}\label{sec:bivariate-examples}

\subsection{General model}
\newtext{For clarity, we use long-form expressions such as $F_{Z_1,Z_2\mid
  V_1,V_2}$ and $p_{\ell_1,\ell_2}(v_1,v_2)$ in the
bivariate case. For $(v_1,v_2) \in \partyU_1 \times \partyU_2$, the multinomial probability function $p_{\ell_1,\ell_2}(v_1,v_2)$ in Proposition~\ref{prop:multinomialp} equals
\begin{multline}\label{eq:1}
    F_{Z_1,Z_2\mid
      V_1,V_2}\left(G_{1}(a_{1\ell_1},v_1), G_{2}(a_{2\ell_2},v_2) \mid v_1, v_2\right) 
  -F_{Z_1,Z_2\mid V_1,V_2}\left(G_{1}(a_{1(\ell_1-1)},v_1),
    G_{2}(a_{2\ell_2},v_2) \mid v_1, v_2\right) \\
  -F_{Z_1,Z_2\mid V_1,V_2}\left(G_{1}(a_{1\ell_1},v_1),
   G_{2}(a_{2(\ell_2-1)},v_2) \mid v_1,v_2\right) \\ +
  F_{Z_1,Z_2\mid V_1,V_2}\left(G_{1}(a_{1(\ell_1-1)},v_1),
    G_{2}(a_{2(\ell_2-1)},v_2) \mid v_1,v_2\right).
\end{multline}
Using Sklar's theorem, $F_{Z_1,Z_2\mid
  V_1,V_2}$ can be written}
via its copula $C_{Z_1,Z_2 \mid
     V_1, V_2}$ and margins $F_{Z_i\mid V_1, V_2}$ as
 \begin{equation}\label{eq:9}
   F_{Z_1,Z_2 \mid V_1, V_2}(z_1,z_2 \mid v_1, v_2) = C_{Z_1,Z_2 \mid
     V_1, V_2}\left(F_{Z_1\mid V_1, V_2}(z_1 \mid v_1, v_2), F_{Z_2
       \mid V_1, V_2}(z_2 \mid v_1, v_2) \mid v_1, v_2\right).
 \end{equation}
 Using the standard notation $H^{(i)}(u_1,u_2) =
 \frac{\partial}{\partial x_i}C(u_1,u_2)$ for partial derivatives
 of a copula, the margins are
 \begin{align*}
   F_{Z_1\mid
   V_1, V_2}(z_1
   \mid v_1,
   v_2) &= H^{(2)}_{Z_1,
          V_2 \mid
          V_1}\left(z_1, H^{(1)}_{V_1,V_2}(v_1,v_2) 
     \mid
          v_1\right)
   \\
   F_{Z_2\mid
   V_1, V_2}(z_2
   \mid v_1,
   v_2) 
        &= H^{(2)}_{Z_2,
          V_1 \mid
          V_2}\left(z_2,  H^{(2)}_{V_1,V_2}(v_1,v_2) 
       \mid
          v_2\right).
 \end{align*}
Thus, in general,
calculation of $\wtfunc(u_1,u_2)$ requires the copula $C_{Z_1,Z_2 \mid
     V_1, V_2}$
and the $H$-functions 
 of the
 copulas $C_{V_1,V_2}$, $C_{Z_1,V_2\mid V_1}$ and $C_{Z_2,V_1
   \mid V_2}$.
 It may be noted that the densities of these four copulas are the ones that appear in a 
D-vine copula decomposition of the joint density $c_{V_1,Z_1,Z_2,V_2}$
of $(V_1,Z_1,Z_2,V_2)$.

 \subsection{Example with v-transforms}\label{sec:example-with-v}

\newtext{We derive the form of the density
$c_{U_1,U_2}(u_1,u_2)$ in a number of cases that are chosen to illustrate
the flexibility of the construction.  We make use of an example}
from~\cite{bib:mcneil-neslehova-smith-25}
in which $T_1(u) =
T_2(u) =|2u-1|$, the symmetric v-transform, and $C_{V_1,V_2}=
C^{\text{Ga}}_{0.85}$, the Gaussian copula with parameter $\rho
=0.85$.
In this case, for $i\in
\{1,2\}$, we have $L_i=2$, $A_{i1} = (0,0.5)$, $A_{i2} = (0.5,1)$, $\partyU_i=(0,1)$ and
$|T_i^\prime(u)| = 2$ for $u \in A_{i1} \cup A_{i2}$. For $v
\in (0,1)$, we have
$G_i(a_{i0},v) =0$, $G_i(a_{i1},v) =0.5$ and $G_i(a_{i2},v) = 1$.

\paragraph{Independent stochastic inversion.}\label{ex:one}
When $(Z_1,Z_2)
\indep (V_1,V_2)$ and $Z_1 \indep Z_2$ we have $\wtfunc(u_1,u_1) = 1$
and  $c_{U_1,U_2}(u_1,u_2) =
c^{\text{Ga}}_{0.85}(|2u_1-1|,|2u_2-1|)$. The contour plot of this
density is in the first panel in Figure~\ref{fig:sipictures}.

\paragraph{$(Z_1,Z_2)$ independent of $(V_1,V_2)$ but $(Z_1,Z_2)$ dependent.}\label{ex:two}
\newtext{In this case $F_{Z_1,Z_2\mid V_1,V_2} = C_{Z_1,Z_2}$ and we choose
$C_{Z_1,Z_2}(z_1,z_2) = C^{\text{Ga}}_{0.8}$. To describe the function $\wtfunc(u_1,u_2)$
it helps to introduce terminology for
the open quadrants of the unit square: $Q_1 = A_{11} \times A_{21} =
(0,0.5)^2$, $Q_2 = A_{11} \times A_{22} = (0,0.5) \times (0.5,1)$, $Q_3
= A_{12} \times A_{21} = (0.5,1) \times (0,0.5)$ and $Q_4 = A_{12}
\times A_{22} = (0.5,1)^2$. It is straightforward to verify that
$\wtfunc(u_1,u_2) = 4 \left(p_{1,1}\Ind{(u_1,u_2) \in Q_1
    \cup Q_4} + (0.5 - p_{1,1}) \Ind{(u_1,u_2) \in Q_2
    \cup Q_3} \right)$, where $p_{1,1} =
C^{\text{Ga}}_{0.8}(0.5,0.5)$. It may be noted that the only 
value of the randomizer model $C_{Z_1,Z_2}$ that enters the density
$\wtfunc$ is the value $p_{1,1}$ at $(z_1,z_2) = (0.5,0.5)$ and $\wtfunc$ describes
a simple patchwork copula on the unit
square. The resulting density $c_{\bm{U}}$ is illustrated as a
contour plot in the second panel of Figure~\ref{fig:sipictures}.}


\begin{figure}[!htb]
  \centering
  \includegraphics[width=4cm,height=4.5cm]{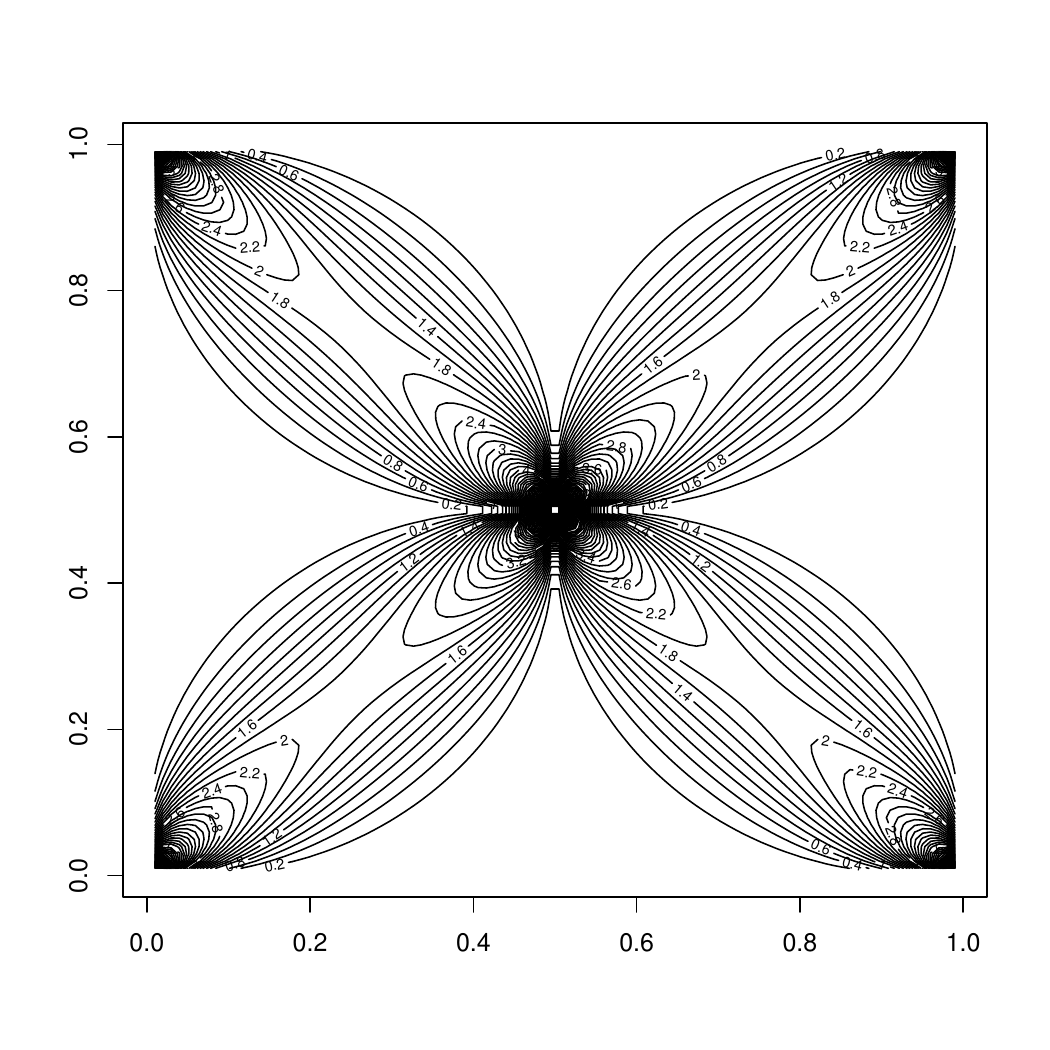}
  \includegraphics[width=4cm,height=4.5cm]{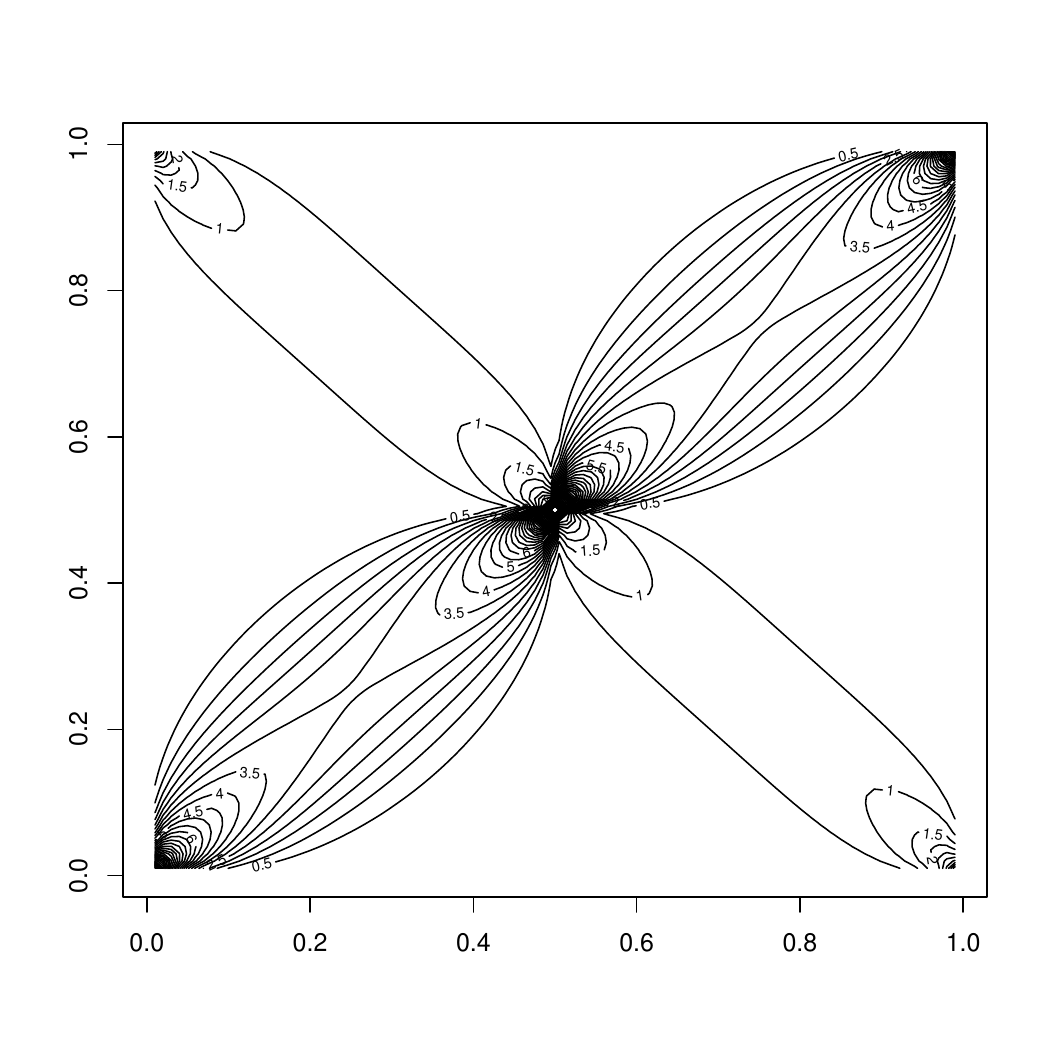}
  \includegraphics[width=4cm,height=4.5cm]{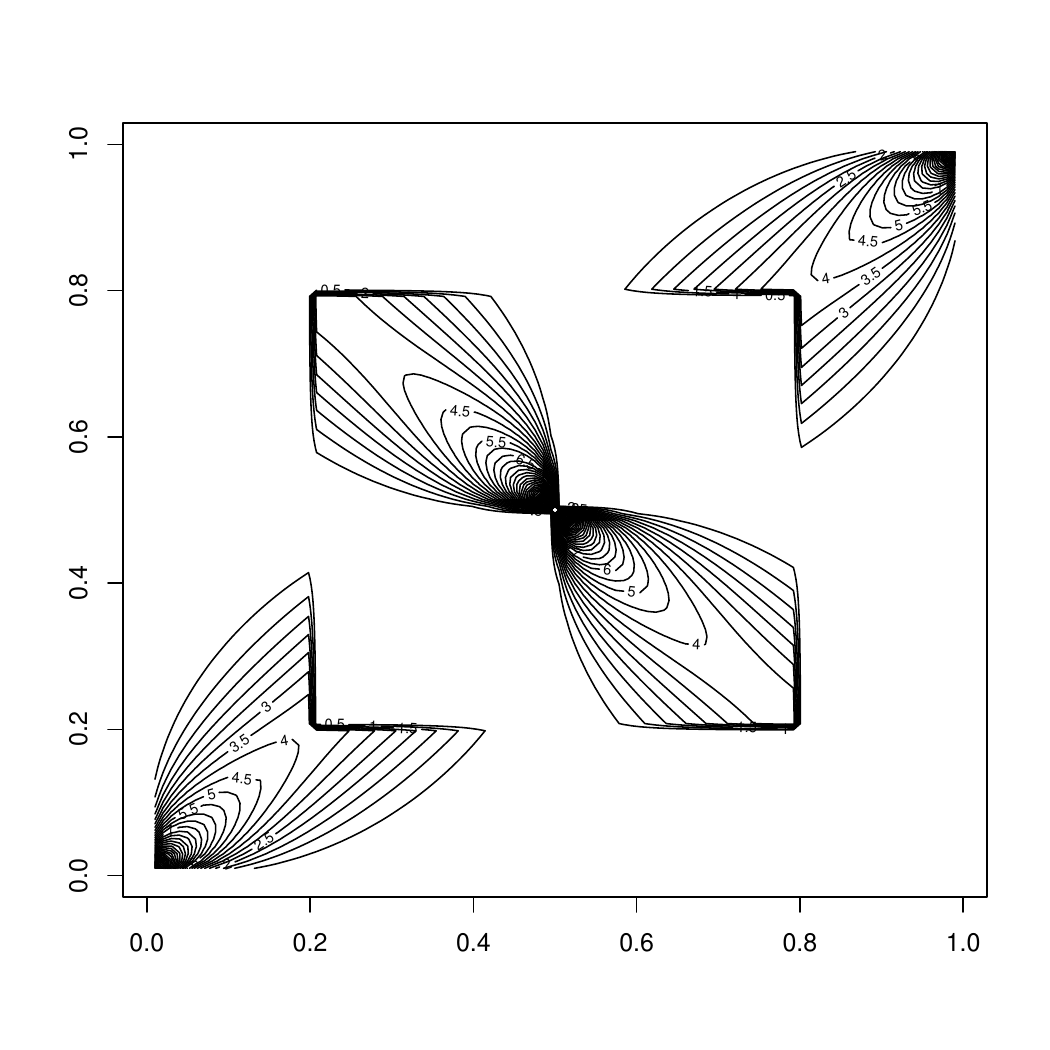}
  \includegraphics[width=4cm,height=4.5cm]{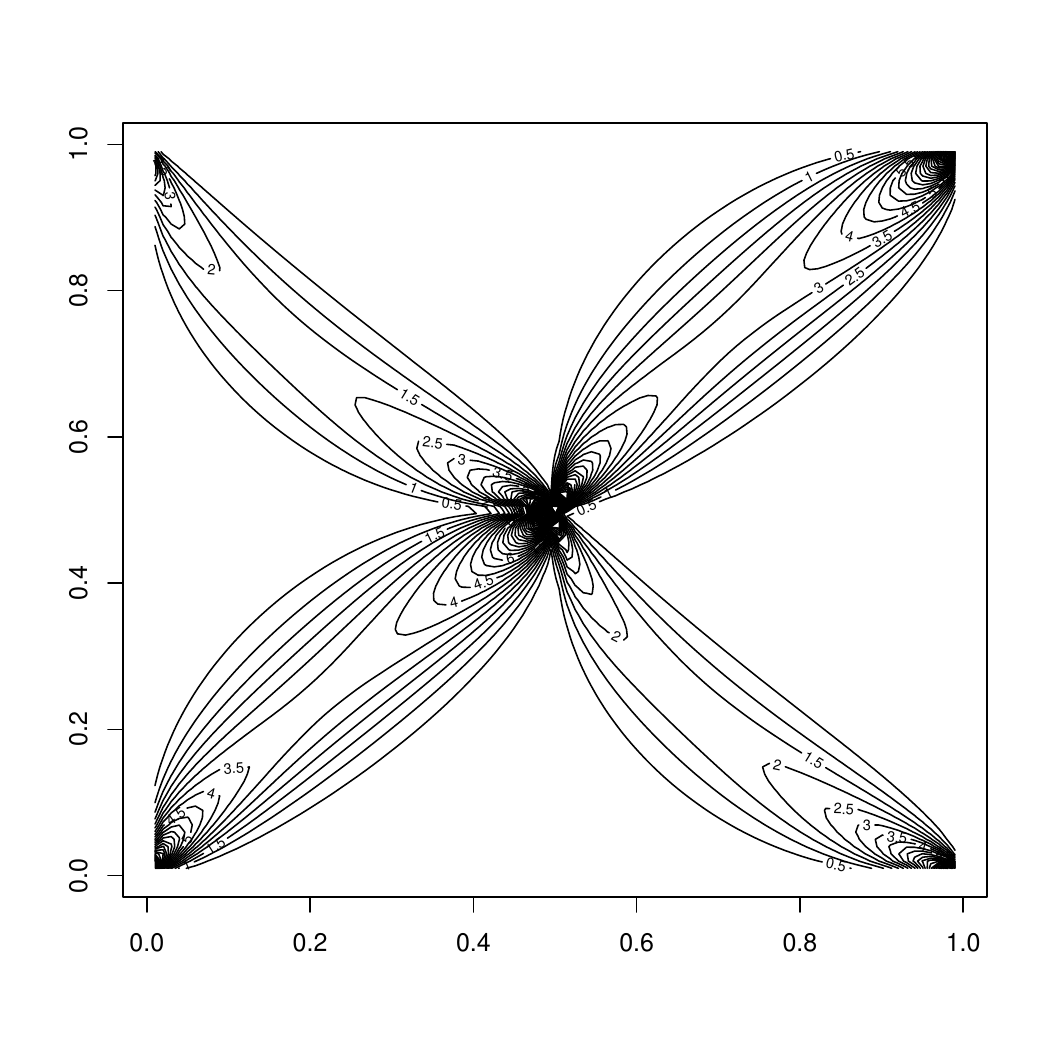}
  \vspace{-0.5cm}
  \caption{Contour plots of the bivariate copula densities in
    Section~\ref{sec:example-with-v}. Componentwise transformation of a random vector
    $(U_1,U_2)$ distributed according to these densities with the
    udp transformation $T(u) = |2u-1|$ results in a random
    vector $(V_1,V_2)$ distributed according to the
    Gaussian copula $C^{\text{Ga}}_{0.85}$. \label{fig:sipictures}}
\end{figure}

 \paragraph{$Z_1$ independent of $(V_1,V_2)$ and
   $Z_2$ independent of $(V_1,V_2)$.}\label{ex:three}
 Although we set $F_{Z_i\mid V_1, V_2}(z \mid v_1, v_2) = F_{Z_i}(z) =
 z$ for $i\in \{1,2\}$, the copula of the
 conditional joint df of
 $(Z_1,Z_2)$ given  $(V_1,V_2)$
 may still depend on the latter so that
 $F_{Z_1,Z_2\mid V_1,V_2} = C_{Z_1,Z_2\mid V_1,V_2}$. For example, let
$C_{Z_1,Z_2\mid V_1,V_2}(z_1,z_2\mid v_1, v_2) =
  M(z_1,z_2) \Ind{\max(v_1,v_2) > k} + W(z_1,z_2) \Ind{\max(v_1,v_2)
    \leq k}$ for $k \in (0,1)$ where $M$ is the comonotonicity copula
  and $W$ is the countermonotonicity copula. If we write
  $S_k = \{(u_1,u_2) : \max(T_1(u_1),T_2(u_2)) \leq k\} = [1/2
  -k/2, 1/2 + k/2]^2$  for a square with side length $k$ and
  $S_k^\complement = (0,1)^2\setminus S_k$ for its complement, we
  obtain that  $\wtfunc(u_1,u_2) = 2 \Ind{(u_1,u_2) \in (Q_1
   \cup Q_4) \cap S_k^\complement} + 2 \Ind{(u_1,u_2) \in (Q_2
   \cup Q_3) \cap S_k}$ 
which is itself a copula density.
The contour plot of the density $c_{U_1,U_2}$ is shown in the third panel
of Figure~\ref{fig:sipictures} for $k = 0.6$ and part of the outline of the
square $S_k$ can clearly be discerned.


 \paragraph{Simplified D-vine.}\label{ex:four}
As is standard with pair copula decompositions, we can
 simplify the general model~\eqref{eq:9} by imposing the so-called
 simplifying assumption whereby the copulas have
   no explicit dependence on the realized values of the conditioning
   variables. Thus,
 $C_{Z_1,Z_2 \mid
     V_1, V_2}(z_1,z_2 \mid v_1,v_2) = C_{Z_1,Z_2 \mid
     V_1, V_2}(z_1,z_2)$, $C_{Z_1,V_2\mid V_1}(z_1,v_2 \mid v_1) =
   C_{Z_1,V_2\mid V_1}(z_1,v_2)$
   and $C_{Z_2,V_1\mid V_2}(z_2,v_1 \mid v_2) =
   C_{Z_2,V_1\mid V_2}(z_2,v_1)$.For example, we choose
$C_{Z_1,Z_2 \mid V_1, V_2} = C^{\text{Ga}}_{\rho_1}$
   $C_{Z_1,V_2 \mid V_1} =
C^{\text{Ga}}_{\rho_2}$ and $C_{Z_2,V_1 \mid V_2} = C^{\text{Ga}}_{\rho_3}$
with $\rho_1 = 0.8$, $\rho_2 = 0.7$ and $\rho_3 = 0.1$.
The contour plot of the
copula density $c_{U_1,U_2}$ is shown in the fourth picture in
Figure~\ref{fig:sipictures}. This is the only example where $F_{\bm{Z}\mid\bm{V}}$ does not
  fulfill the conditions in
Theorem~\ref{prop:densityresult} for the probability density
$\wtfunc$ to be a copula density. More details of $\wtfunc$
and the corresponding
multinomial allocation function $p_{\ell_1,\ell_2}(v_1,v_2)$ are given in Appendix~\ref{appendix:C}.

\section{An application}\label{sec:applications}

\newtext{Our main result,  Theorem~\ref{theorem:mult-stoch-invers}, gives a very
  general form for $c_{\bm{U}}$ that depends on three elements: the udp functions
  $T_1,\ldots,T_d$, the base copula model $c_{\bm{V}}$; the
  multinomial probability function
  $p_{\ell_1,\ldots,\ell_d}(\bm{v})$. We outline our proposed  inference strategy for data on $\bm{U}$ before
  giving a bivariate application where all three elements of the model
  are important.}

\newtext{The key step is selection of functions $T_1, \ldots,T_d$ that
render the components of $\bm{V} = \mathcal{T}(\bm{U})$ strongly
monotonically dependent. In some
applications, such as our example in this section, the
form of these functions may
be suggested by the structure of the data; in other cases an estimation
procedure like ACE~\citep{bib:breiman-friedman-85} could be used. If $T_1,\ldots, T_d$ have been fixed, the 
log-likelihood for the model separates into a term that depends 
on $c_{\bm{V}}$ and a term that depends on
$p_{\ell_1,\ldots,\ell_d}(\bm{v})$, making separate
optimization possible. The latter term is 
difficult to model directly and we argue in Section~\ref{sec:mult-stoch-invers} that it should be
parameterized via a relatively parsimonious randomizer model
$F_{\bm{Z} \mid \bm{V}}$. The parsimony should reflect the fact that $c_{\bm{U}}$ depends only on the
values of $F_{\bm{Z} \mid \bm{V}}$ on a finite set of points. The simplified D-vine model for $F_{\bm{Z} \mid \bm{V}}$  in
Section~\ref{sec:example-with-v} could be a flexible candidate for
this role in many applications.}

\newtext{Our application
  builds on insights in~\citet{bib:dias-han-mcneil-25} for financial
  time series data. In their
    work, the independent stochastic
inversion construction with linear v-transforms is used to construct
D-vine models that can rival classical GARCH processes.
Linear v-transforms are udp functions that generalize the function $T(u) = |1-2u|$ in Figure~\ref{fig:Tplots} by introducing a
parameter $\delta \in (0,1)$ and setting $T(u; \delta) = (\delta-u)/\delta$ for $u \leq \delta$ and
$T(u; \delta) = (u-\delta)/(1-\delta)$ for $u > \delta$, so that
$T(u; 0.5) = |2u-1|$.}

\newtext{Given a time series $X_1,\ldots,X_n$ of log-returns on a financial asset
price,
\citet{bib:dias-han-mcneil-25}  assume strict stationarity and use a
parametric copula density 
$c_{V_1,V_2}(T(u_1;\delta_1), T(u_2;\delta_2); \theta)$ to model the serial
dependence of the pairs $\{(X_t,X_{t+1}) : t=1,\ldots,n-1\}$. They
choose $c_{V_1,V_2}(v_1,v_2;\theta)$ to be the density of an upper
tail-dependent copula with single parameter $\theta$, such as the Joe copula. In the higher levels of
the D-vine tree they 
use
similar bivariate models for the partial
dependence of pairs $(X_t, X_{t+k})$ conditional on $X_{t+1},\ldots,X_{t+k-1}$
for $k > 1$, but we focus solely on $(X_t,X_{t+1})$.}
\newtext{Figure~\ref{fig:fin-example} shows 10 years of daily log-returns for
  the JPMorgan Chase stock price including the 2008 financial
  crisis. The empirical copula data $\{(U_t,U_{t+1}): t=1,\ldots,n-1\}$,
  constructed using the empirical df of the series, are
  also shown.}

\newtext{Three nested models based on the Joe copula for $c_{V_1,V_2}$
  are fitted to these data using a full maximum-likelihood procedure
  for all parameters; since $T(u_1;\delta_1)$ and
  $T(u_2;\delta_2)$ take a relatively simple form, we do not fix $\delta_1$ and
  $\delta_2$ in an initial step. 
  In the results of Table~\ref{tab:pars} the models are as follows: Model 1 is independent
  stochastic inversion; Model 2 follows the second construction in
  Section~\ref{sec:example-with-v} with
$C_{Z_1,Z_2} = C^{\text{Ga}}_{\rho_1}$; Model 3 follows
the simplified D-vine in Section~\ref{sec:example-with-v}. Model 2 is
favoured since it improves significantly on Model 1 ($p=0.0002$ in 
likelihood-ratio test), whereas Model 3 does not improve significantly
on Model 2 ($p=0.39$).
The estimate of $\rho_1$ indicates
negative dependence between $Z_1$ and $Z_2$, which has a particularly interesting
interpretation. $Z_t$ represents the direction of a market movement
while $V_t$ represents its size, so our results
suggest a statistically significant tendency for downward movements to be followed by upward
movements, and vice versa, contrary to the view that efficient markets
preclude directional forecasts.
}

\begin{table}[ht]
  \centering
  \begin{tabular}{l@{\hspace{4pt}}c@{\hspace{4pt}}r@{\hspace{2pt}}lr@{\hspace{2pt}}lr@{\hspace{2pt}}lr@{\hspace{2pt}}lr@{\hspace{2pt}}lr@{\hspace{2pt}}l} \toprule
   Model & log-lik. & $\delta_1$ & & $\delta_2$ & & $\theta$ & &
                                                                 $\rho_1$
    & & $\rho_{2}$ & & $\rho_3$&
    \\ \midrule
    1 & 163.7 & 0.65 & (0.02) & 0.47 &
                                                                      (0.02)
                                                        & 1.37 &
                                                                 (0.03)
                                                                     &
                                                                        &\\
    2 & 170.4 & 0.62 & (0.003) &
                                                                       0.50
                                                                       &
                                                                         (0.002)
                                                        & 1.39 & (0.03)
                                                                         &
                                                                           -0.11&
                                                                           (0.03)
      & & & &
    \\
    3 & 171.4 & 0.62& (0.002) & 0.50 & (0.002) & 1.38 & (0.03) & -0.12
    & (0.03) & -0.0004 & (0.03) & 0.03 & (0.03) \\
    \bottomrule
  \end{tabular} 
  \caption{\newtext{Parameter estimates (standard errors) for models fitted to
    daily log-return data for JPMorgan.}}
  \label{tab:pars}
\end{table}
\vspace{-0.3cm}


\begin{figure}[!htb]
  \centering
  \includegraphics[width=8cm,height=4cm]{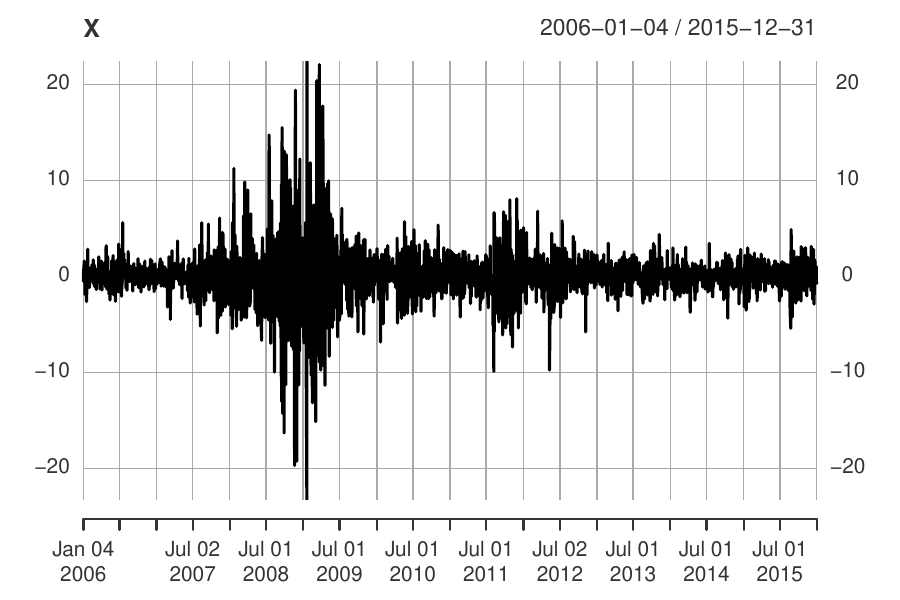}
  \includegraphics[width=4cm,height=4cm]{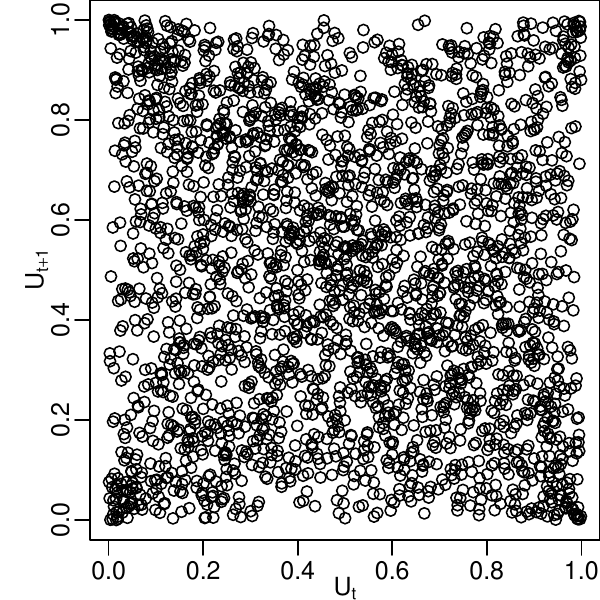}
   \includegraphics[width=4cm,height=4cm]{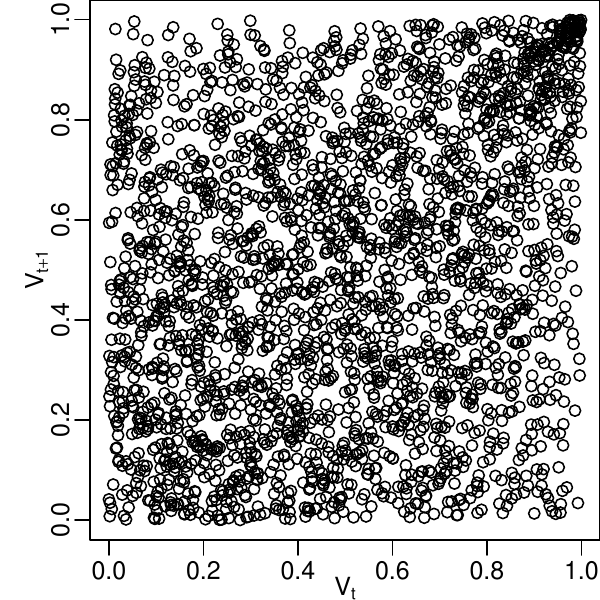}
  \vspace{-0.2cm}
  \caption{Left: time series of 2498 daily log-returns of JPMorgan
    Chase 2006-2015. Middle: empirical copula data $(U_t,U_{t+1})$. Right: data $(V_t,V_{t+1})
    = (T(U_t; \delta_1),T(U_{t+1};\delta_2))$ which are modelled
    by a Joe copula with parameter $\theta$; values of $\delta_1$ and
    $\delta_2$ are taken from independent stochastic inversion model. \label{fig:fin-example}}
\end{figure}

\newtext{Further development of the methodology for
datasets showing non-monotonic dependencies that are more complex than
those that can be modelled with v-transforms, as well as extensions
  to higher dimensions, are subjects for further research.}



\newcommand{\noopsort}[1]{}

\section*{Reproducibility}

\noindent All examples have been created with the \texttt{udp} \textsf{R} library at \texttt{github.com/ajmcneil/udp}. Code and data are found in the repository \texttt{github.com/ajmcneil/papers}.

\appendix
\renewcommand{\thesection}{\Alph{section}}

\section{Construction of the motivating example}\label{appendix:A}

Recall that, in each of the first three columns of Figure~\ref{fig:motivation}, $X$ and $Y$ have standard
normal distribution functions $\Phi$. We have constructed examples
such that $g$ and $h$
are piecewise continuous and strictly monotonic with 
finitely many pieces and $g(X) \eqdis h(Y)
\sim \chi^2_1$, i.e.\ a chi-squared distribution with one degree of
freedom. Writing $F_{\chi^2_1}$ for the df of this distribution, the udp transformations mapping $U_1=
\Phi(X)$ into $V_1= F_{\chi^2_1}(g(X))$ and  $U_2=
\Phi(Y)$ into $V_2= F_{\chi^2_1}(h(Y))$ are
\begin{displaymath}
  T_1(u) = F_{\chi^2_1} \circ g \circ \Phi^{-1}(u)\quad\text{and}\quad T_2(u) = F_{\chi^2_1} \circ h \circ \Phi^{-1}(u), \quad u \in [0,1].
\end{displaymath}
The examples are created by
fixing udp transformations $T_1$ and $T_2$ so that $g$ and $h$ are implicitly given by
\begin{displaymath}
  g(x) = F_{\chi^2_1}^{-1} \circ T_1 \circ \Phi(x)\quad\text{and}\quad  h(y) = F_{\chi^2_1}^{-1} \circ T_2 \circ \Phi(x),\quad x,y \in \R.
\end{displaymath}

In the first column of Figure~\ref{fig:motivation} $T_1$ and $T_2$ are
the symmetric v-transform $T_1(u) = T_2(u) = |1-2u|$, implying
$$
g(x)=h(x) = F_{\chi^2_1}^{-1} \big(|1-2 \Phi(x)|\big) = x^2,
$$
where the final equality follows by noting that  $F_{\chi^2_1}^{-1}(y) = \Phi^{-1}((1+y)/2)^2$ for
$y \in [0,1)$. Thus $g$ and $h$ are continuous and strictly monotonic
with two pieces.

In the second column of Figure~\ref{fig:motivation} $T_1(u) = 1- \pi^{-1}\arccos((- \cos(3\pi u))$ and $T_2(u) = u$, the
identity function on $[0,1]$. The function $T_1$ is the simplest
extension of the symmetric v-transform to 3 linear pieces and is
shown in the left panel of Figure~\ref{fig:supplement}. The implied function $g$ does not have
a simple closed form in this case but it can be easily
plotted and is seen to be piecewise continuous and strictly monotonic with three
distinct pieces as shown in the right panel of
Figure~\ref{fig:supplement}.

\begin{figure}[t]
   \centering
   \includegraphics[width=12cm,height=5.5cm]{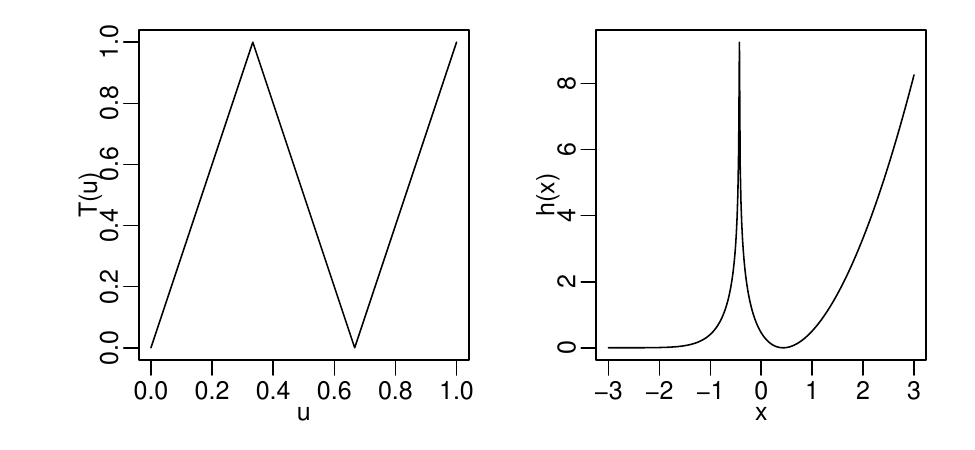}\\
  \vspace{-0.5cm}
   \caption{\label{fig:supplement}  Left: the udp function $T(u) = 1-
     \pi^{-1}\arccos((- \cos(3\pi u))$. Right: the corresponding function $g$
     determined by the equation $g(x) = F_{\chi^2_1}^{-1} \circ T \circ \Phi(x)$.}
 \end{figure}

In the third
column of Figure~\ref{fig:motivation} we have taken $T_1(u) = T_2(u) =
1- \pi^{-1}\arccos((- \cos(3\pi u))$. We further remark that this udp
transformation and the symmetric v-transform $|1-2u|$ are members of a
family of udp functions induced by an orthonormal basis of cosine
functions on the unit interval $[0,1]$ as shown in~\cite{bib:mcneil-neslehova-smith-25}.

\section{Proof of Theorem 3}\label{appendix:B}

We recall that Theorem 3 asserts that the function
$\wtfunc(u_1,\ldots,u_d)$ in equation~\eqref{eq:4} is a probability density
function. Moreover, if $Z_i$ is independent of $\bm{V}$ for $i=1,\ldots,d$, it
is the density of a copula.

\begin{proof} Exchanging the order of integration and summation
  allows us to write 
  \begin{displaymath}
  \int_0^1 \cdots \int_0^1\wtfunc(\bm{u})\rd\bm{u} = \sum_{\ell_1=1}^{L_1}
                                                         \cdots
                                                         \sum_{\ell_d=1}^{L_d}
                                                         \int_{A_{1\ell_1}}
                                                         \cdots
                                                         \int_{A_{d\ell_d}}
                                                        p_{\ell_1,\ldots,\ell_d}(\mathcal{T}(\bm{u}))
                                                        \prod_{i=1}^d
                                                         |T_i^\prime(u_i)|
                                                       \rd \bm{u}.
  \end{displaymath}
  Since
  $T_i\vert_{A_{i\ell_i}}$ is strictly monotonic and continuously
  differentiable, we can make
  the substitution $v_i = T_i(u_i)$ on $A_{i\ell_i}$ which gives
\begin{align*}
  \int_0^1 \cdots \int_0^1\wtfunc(\bm{u})\rd\bm{u}  &=
        \sum_{\ell_1=1}^{L_1}
            \cdots
                                                         \sum_{\ell_d=1}^{L_d}
                                                         \int_{R_{1\ell_1}}
                                                         \cdots
                                                         \int_{R_{d\ell_d}} p_{\ell_1,\ldots,\ell_d}(\bm{v})
                                                       \rd \bm{v} =
        \sum_{\ell_1=1}^{L_1}
            \cdots
                                                         \sum_{\ell_d=1}^{L_d}
                                                         \int_{0}^1
                                                         \cdots
                                                         \int_{0}^1p_{\ell_1,\ldots,\ell_d}(\bm{v})
                                                       \rd \bm{v},
\end{align*}
where the second expression follows because $p_{\ell_1,\ldots,\ell_d}(\bm{v})=0$ for $\bm{v}
     \not\in R_{1\ell_1} \times \cdots \times R_{d\ell_d}$. Finally,
   exchanging the order of summation and integration again gives the
   result, since the probabilities
$p_{\ell_1,\ldots,\ell_d}(\bm{v})$  sum to one over all possible
cells, except on a null set in $[0,1]^d$.

Now we assume that $Z_i$ is independent of $\bm{V}$ for $i=1,\ldots,d$
and show, without loss of generality, that $\wtfunc_1(u_1) := \int_{u_2=0}^1 \cdots
\int_{u_d=0}^1\wtfunc(\bm{u})\rd\bm{u} =1$ for all $u_1 \in \unionA_1$. Using a similar sequence
of operations to the ones above we can show that
 \begin{align*}
 \wtfunc_1(u_1) &= \sum_{\ell_1=1}^{L_1} \Ind{u_1 \in A_{1\ell_1}}\sum_{\ell_2=1}^{L_2}
                                                         \cdots
                                                         \sum_{\ell_d=1}^{L_d}
                                                         \int_{A_{2\ell_2}}
                                                         \cdots
                                                         \int_{A_{d\ell_d}}
                                                        p_{\ell_1,\ldots,\ell_d}(\mathcal{T}(\bm{u}))
                                                         \prod_{i=1}^d
                                                         |T_i^\prime(u_i)|
                                                       \rd u_2 \cdots
                  \rd u_d \\
   &=\sum_{\ell_1=1}^{L_1} \Ind{u_1 \in A_{1\ell_1}} \sum_{\ell_2=1}^{L_2}
                                                         \cdots
                                                         \sum_{\ell_d=1}^{L_d}
                                                         \int_{0}^1
                                                         \cdots
                                                         \int_{0}^1
                                                         \frac{p_{\ell_1,\ldots,\ell_d}\left(T_1(u_1),v_2,\ldots,v_d
     \right)
                                                         }{
                                                         |T_1^\prime(u_1)|^{-1}}
                                                       \rd v_2 \cdots
     \rd v_d \\
   &= \sum_{\ell_1=1}^{L_1} \Ind{u_1 \in A_{1\ell_1}}  |T_1^\prime(u_1)|   \int_{0}^1
                                                         \cdots
                                                         \int_{0}^1 \sum_{\ell_2=1}^{L_2}
                                                         \cdots
                                                         \sum_{\ell_d=1}^{L_d} p_{\ell_1,\ldots,\ell_d}\left(T_1(u_1),v_2,\ldots,v_d
     \right)    \rd v_2 \cdots
     \rd v_d.
 \end{align*}
 Now the integrand may be written as
 \begin{multline*}
 \sum_{\ell_2=1}^{L_2}
                                                         \cdots
                                                         \sum_{\ell_d=1}^{L_d} p_{\ell_1,\ldots,\ell_d}\left(T_1(u_1),v_2,\ldots,v_d
                                                         \right) \\
                                                         =  \P\left( G_{1}(a_{1(\ell_1-1)},v_1) < Z_1 \leq
      G_1(a_{1\ell_1},v_1) \mid \bm{V} = (T_1(u_1),v_2,\ldots,v_d)^\top\right)
   \end{multline*}
   and, by the independence of $Z_1$ and $\bm{V}$, this is simply
   equal to $p_{1\ell_1}(T_1(u_1)) = |T_1^\prime(u_1)|^{-1}$.
   The assertion now follows easily.
    \end{proof}

    The only example in Section~\ref{sec:example-with-v} where $\wtfunc(u_1,u_2)$ is not
    a copula density is the simplified D-vine example illustrated in the final
    panel of Figure~\ref{fig:sipictures}.
In this example the function $\wtfunc(u_1,u_2)$ takes a very unusual form as
shown in the perspective plot in Figure~\ref{fig:oddhfunc}. We can apply the
arguments used in the proof of the second part of
Theorem~\ref{prop:densityresult} to see that the first marginal
density satisfies
\begin{displaymath}
  \wtfunc_1(u) = \begin{cases}
    2 \int_0^1
                 \P\big( Z_1 \leq 0.5 \mid V_1 =
                \lvert 1-2u \rvert, V_2 = v \big) \rd v, & \quad u \in
                (0,0.5), \\
       2 \int_0^1
                 \P\big( Z_1 >  0.5 \mid V_1 =
                \lvert 1-2u \rvert, V_2 = v \big) \rd v, & \quad u \in
                (0.5,1).
                \end{cases}
\end{displaymath}
This depends on the copulas $C_{Z_1,V_2 \mid V_1}$ and
$C_{V_1,V_2}$ used to specify the D-vine through their
H-functions. The density $\wtfunc_1(u)$ is also shown in
Figure~\ref{fig:oddhfunc} and is clearly not uniform.



\begin{figure}[htb]
  \centering
  \includegraphics[width=11.5cm, height=6cm]{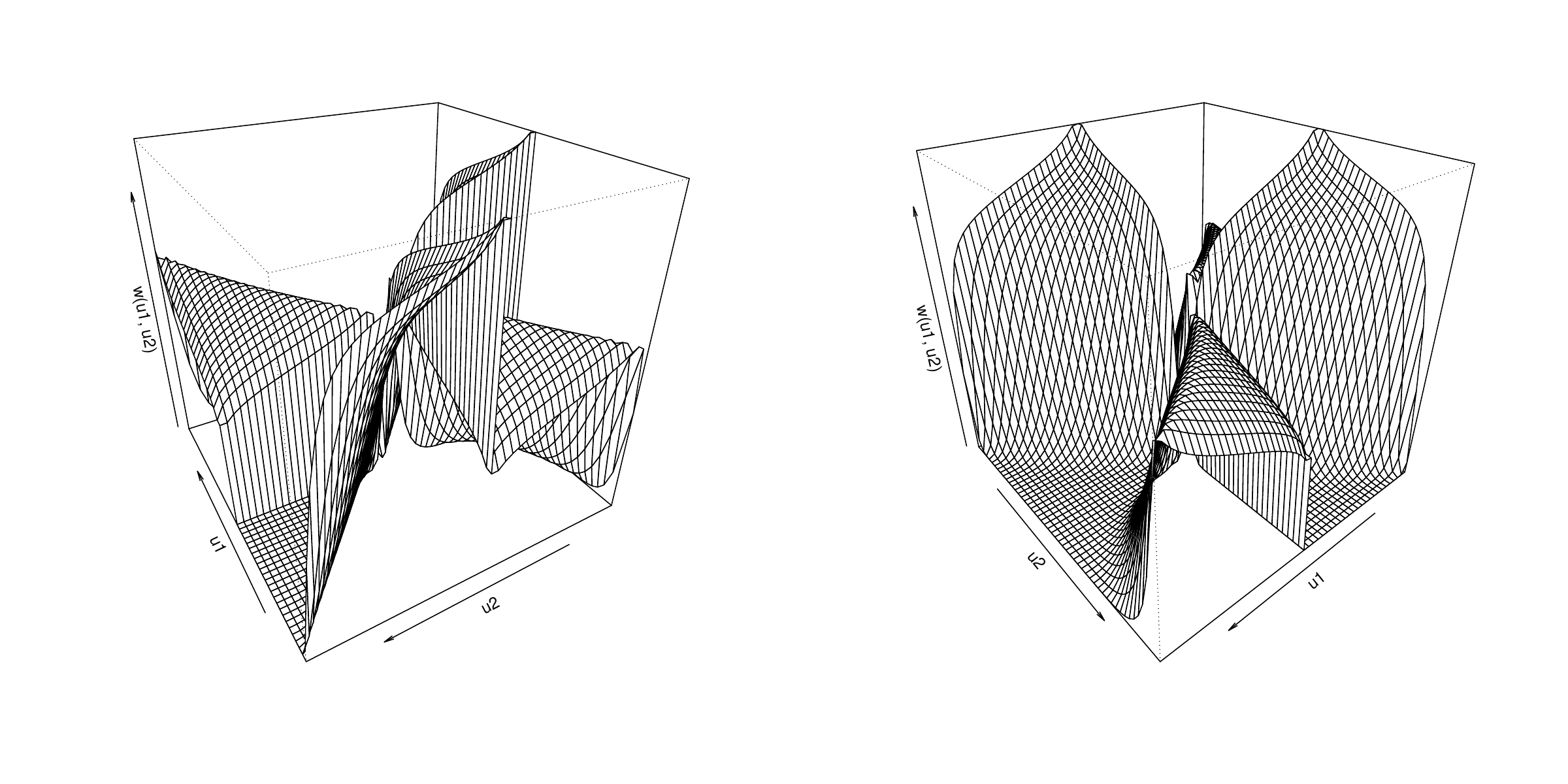}
    \includegraphics[width=4.8cm, height=5.7cm]{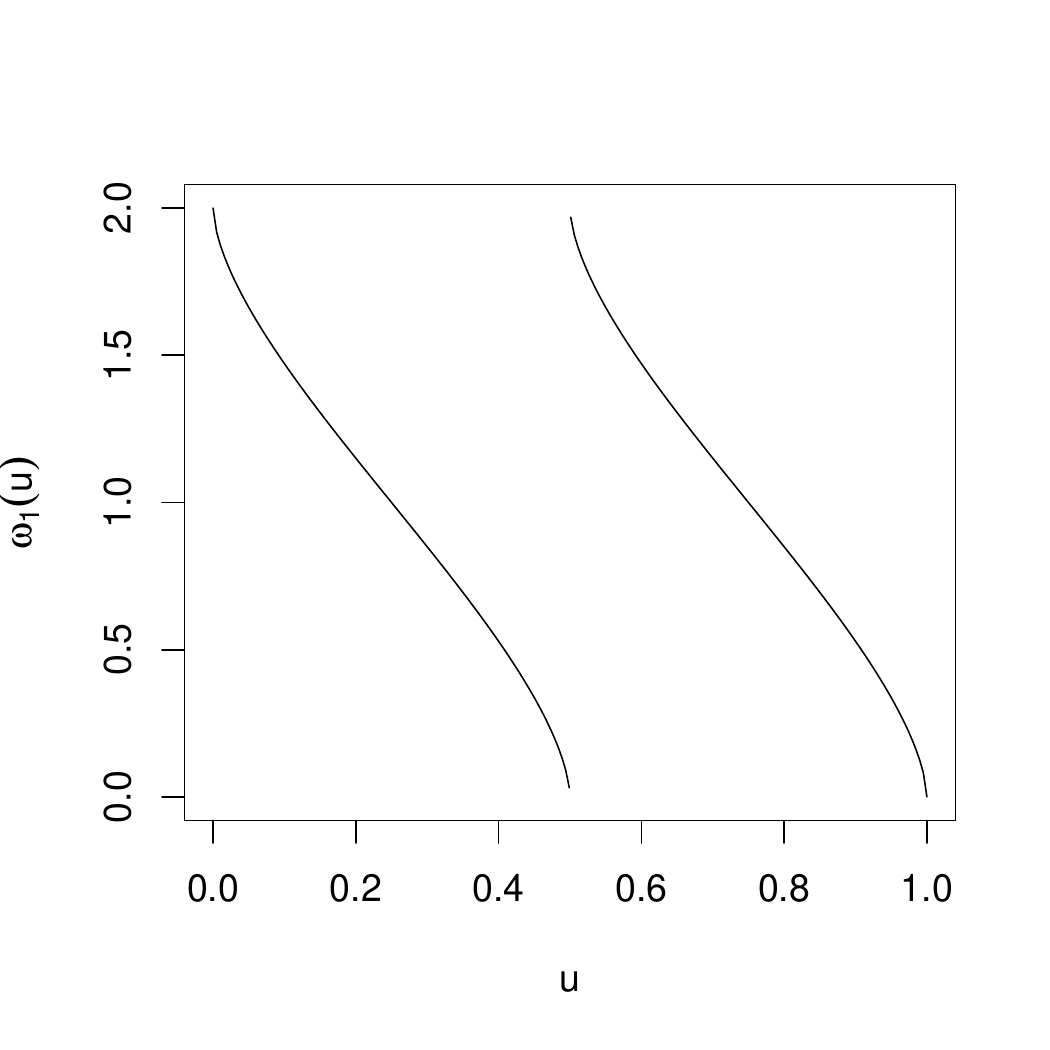}
  \vspace{-0.2cm}
  \caption{Perspective plots of $\wtfunc(u_1,u_2)$ from two different angles and plot of
    corresponding marginal density
    $\omega_1(u)$.}
  \label{fig:oddhfunc}
\end{figure}

\section{Multinomial allocation in the simplified D-vine model}\label{appendix:C}

 \begin{figure}[ht!]
  \centering
  \includegraphics[width=7.5cm, height=7.5cm]{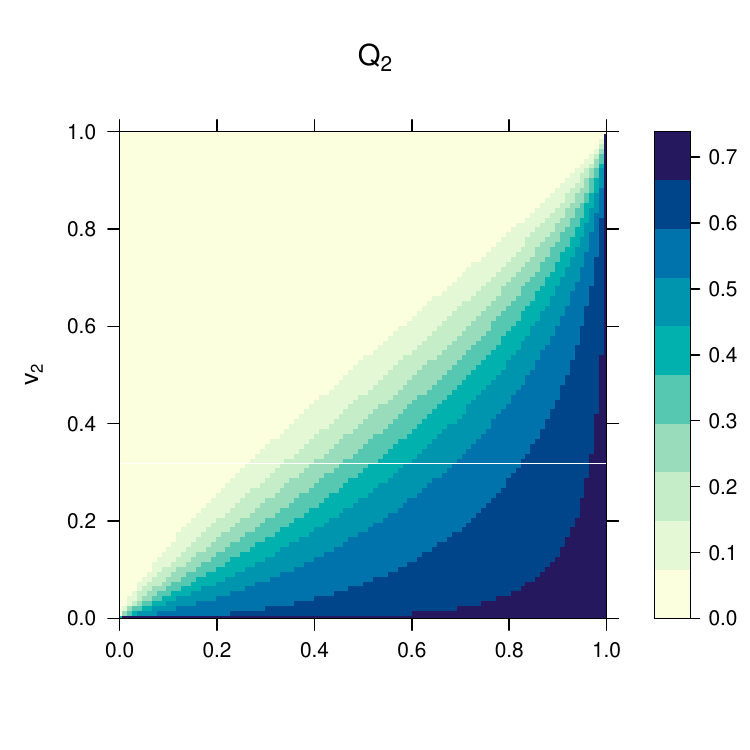}
  \includegraphics[width=7.5cm, height=7.5cm]{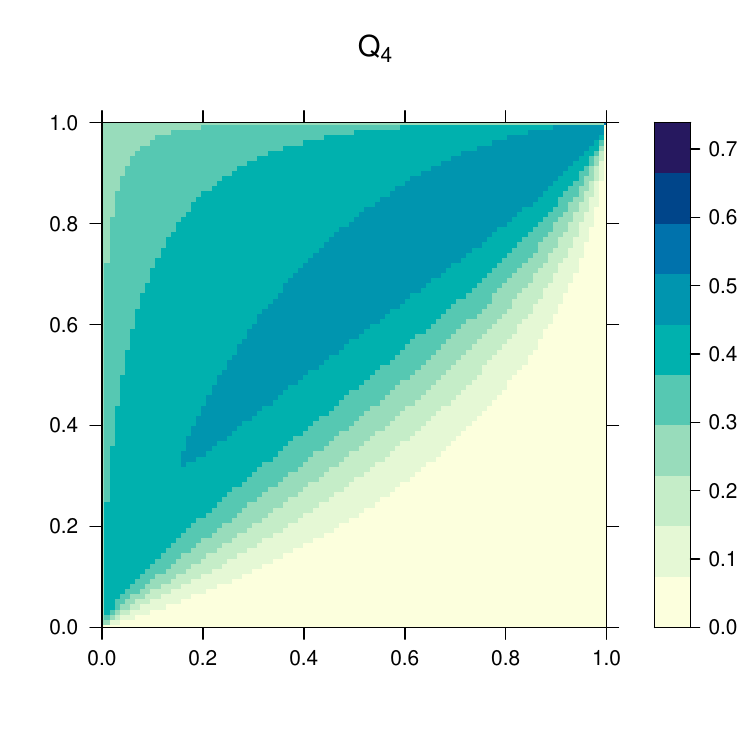} \\
   \includegraphics[width=7.5cm, height=7.5cm]{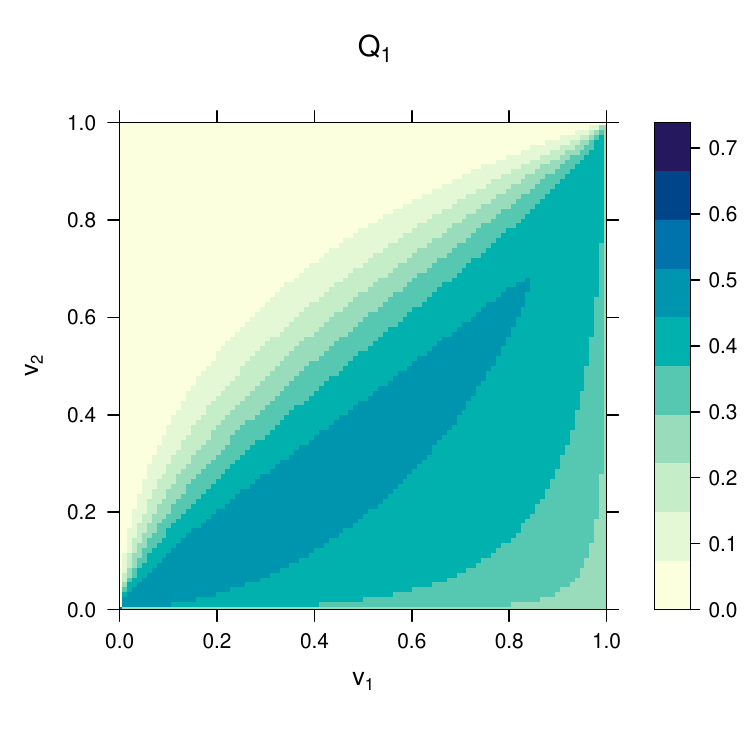}
  \includegraphics[width=7.5cm, height=7.5cm]{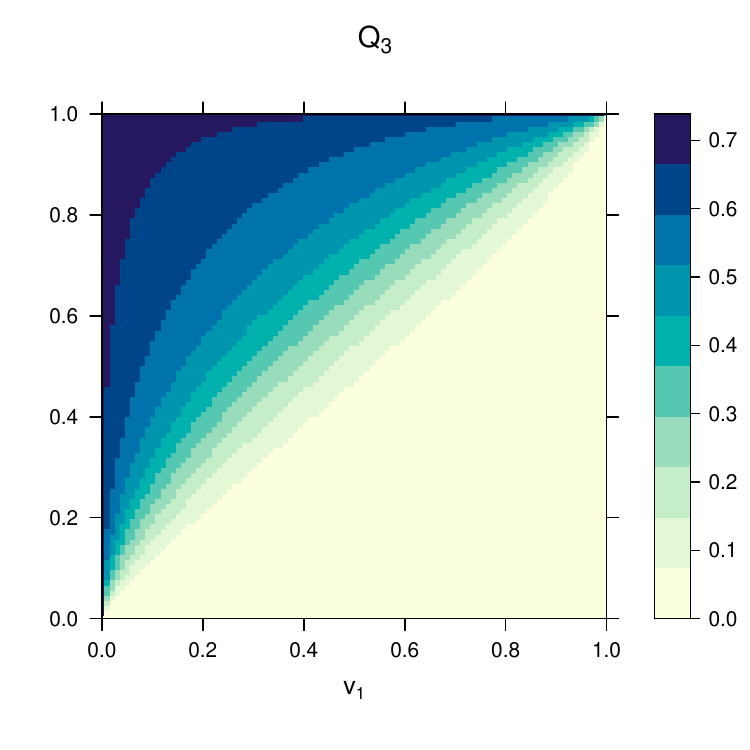} 
  \vspace{-0.2cm}
  \caption{Level plots of the multinomial probability function
    $p_{\ell_1,\ell_2}(v_1v_2)$ using colour scale for size of probabilities. $Q_1$ corresponds to
    $\ell_1=\ell_2=1$; $Q_2$ corresponds to $\ell_1=1$ and $\ell_2=2$;
    $Q_3$ corresponds to $\ell_1=2$ and $\ell_2=1$;
    $Q_4$ corresponds to $\ell_1=\ell_2=2$.}
  \label{fig:levelplots}
\end{figure}

In this section we use the notation from
Section~\ref{sec:example-with-v} for the open
quadrants of the unit square, so that $Q_1 = A_{11} \times A_{21} =
(0,0.5)^2$, $Q_2 = A_{11} \times A_{22} = (0,0.5) \times (0.5,1)$, $Q_3
= A_{12} \times A_{21} = (0.5,1) \times (0,0.5)$ and $Q_4 = A_{12}
\times A_{22} = (0.5,1)^2$. The multinomial
probability function of the simplified D-vine model in
Section~\ref{sec:example-with-v} is
$$
p_{\ell_1,\ell_2}(v_1,v_2) = \P (U_1\in A_{1\ell_1} , U_2 \in
A_{2\ell_2} \mid V_1 = v_1, V_2 = v_2),\quad \ell_1,\ell_2 \in
\{1,2\},\;v_1,v_2 \in (0,1),
$$
and describes the conditional probability of allocation to each of
the quadrants $Q_1,\ldots,Q_4$, given $V_1 = v_1$ and $V_2 = v_2$.
In Figure~\ref{fig:levelplots} we display
the quadrant allocation probabilities as a function of $v_1$ and $v_2$
in the form of level plots. If stacked on top of one another, these
level plots would sum to one.
Broadly speaking, we see that, when the realized values of $V_1$ and
$V_2$ are both small, this leads to a high probability of allocation
to $Q_1$. When the realized 
values of $V_1$ and $V_2$ are both large, this leads to a high probability of allocation to
$Q_4$. Large $V_1$ and small $V_2$ imply a high probability of
allocation to  $Q_2$, while large $V_2$ and small $V_1$ imply a high probability of
allocation to  $Q_3$.

We can also calculate the quadrant
allocation probabilities conditional on only one of the two variables
$V_1$ and $V_2$. For example, to condition only on $V_1$ we have to
integrate $v_2$ out of $p_{\ell_1,\ell_2}(v_1,v_2)$ by calculating
$\int_0^1 p_{\ell_1,\ell_2}(v_1,v_2) f_{V_2 \mid V_1}(v_2 \mid v_1)
\rd v_2$. Since the conditional density $f_{V_2 \mid V_1}(v_2 \mid v_1)$ is identical to the copula
density $c_{V_1,V_2}(v_1,v_2)$, we obtain the expressions
\begin{align*}
p_{\ell_1,\ell_2 \mid 1}(v) & = \P (U_1\in A_{1\ell_1} , U_2 \in
A_{2\ell_2}  \mid V_1 = v) = \int_0^1 p_{\ell_1,\ell_2}(v,x)
                         c_{V_1,V_2}(v,x) \rd x, \\
  p_{\ell_1,\ell_2\mid 2}(v) & = \P (U_1\in A_{1\ell_1} , U_2 \in
A_{2\ell_2}  \mid V_2 = v) = \int_0^1 p_{\ell_1,\ell_2}(x,v)
                         c_{V_1,V_2}(x,v) \rd x.
\end{align*}
These must satisfy the marginal constraints
\begin{equation}\label{eq:marginal-constraint}
\begin{aligned}
      p_{\ell,1\mid 1}(v) + p_{\ell,2\mid 1}(v) &= \P (U_1 \in A_{1\ell} \mid V_1 = v) =p_{1\ell}(v) = 
            0.5,
                \\
     p_{1,\ell \mid 2}(v) + p_{2,\ell \mid 2}(v) &= \P (U_2 \in
     A_{2\ell} \mid V_2 = v) =p_{2\ell}(v) = 0.5,
               \end{aligned}
               \end{equation}
for $\ell \in \{1,2\}$ and $v \in (0,1)$. These constraints simply say that the model must respect the properties of
the univariate stochastic inverse of the symmetric v-transform. In other words, conditional on $V_i$, the probability 
that $U_i$ is assigned to
the set $(0,0.5)$ and the probability that it is assigned to the set $(0.5,1)$ must be equal.
If one were to approach the parameterization of the weight function
$\wtfunc$ by attempting to write down a multinomial regression model for
$p_{\ell_1,\ell_2}(v_1,v_2)$, it would be quite challenging to ensure that these constraints were
fulfilled. By specifying a model for $F_{\bm{Z}\mid\bm{V}}$ that
respects the independence of $Z_i$ and $V_i$ for $i \in \{1,2\}$, the
constraints are satisfied automatically.

Figure~\ref{fig:levelplots2} shows the functions $p_{\ell_1,\ell_2\mid
  i}(v)$ for $i\in\{1,2\}$ as stacked area plots.
The relative probability of allocation to $Q_1$ 
decreases with both $V_1$ and $V_2$ while the relative probability of allocation
to $Q_4$ increases with both $V_2$ and $V_2$. The magnitude of this
effect is stronger for $V_2$ than $V_1$. This is due to the stronger
partial dependence between $Z_1$ and $V_2$ given $V_1$ than between
$Z_2$ and $V_1$ given $V_2$. The arrangement of the area plots makes
clear that the equations~\eqref{eq:marginal-constraint} are indeed satisfied.

\begin{figure}[htb]
  \centering
  \includegraphics[width=7.5cm, height=8.5cm]{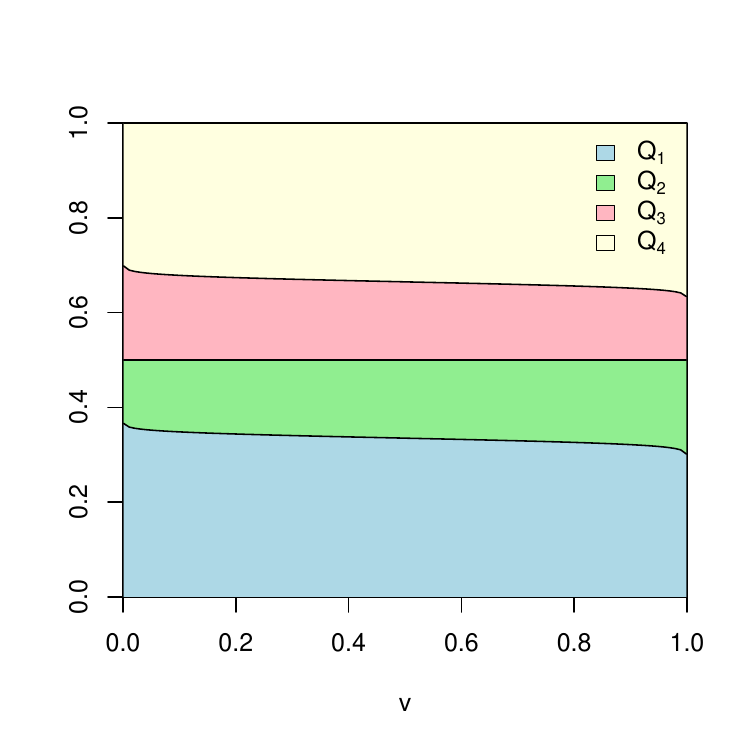}
  \includegraphics[width=7.5cm, height=8.5cm]{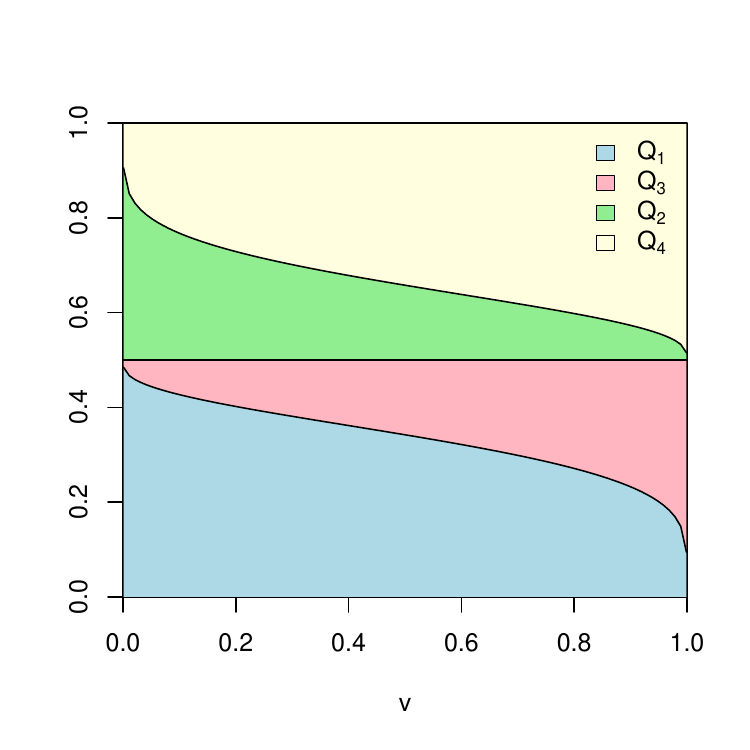} 
  \vspace{-0.2cm}
  \caption{Stacked area plots for the functions $p_{\ell_1,\ell_2\mid
      1}(v)$ (left) and $p_{\ell_1,\ell_2\mid 2}(v)$ (right).}
  \label{fig:levelplots2}
\end{figure}


\end{document}